\documentclass[12pt]{article}

\usepackage{amsmath,amsthm,amssymb,amsfonts}
\usepackage{hyperref}
\usepackage[margin=1.15in]{geometry}
\usepackage{setspace}
\usepackage{yfonts}

\usepackage[T1]{fontenc}\usepackage[upright]{fourier}\usepackage{baskervald}

\usepackage{dsfont}

\newenvironment{frcseries}{\fontfamily{frc}\selectfont}{}
\newcommand{\textfrc}[1]{{\frcseries#1}}
\newcommand{\mathfrc}[1]{\text{\textfrc{#1}}}

\usepackage[all]{xy}

\usepackage{tikz}
\usetikzlibrary{positioning,chains,fit,shapes,calc}
\usepackage{tikz-cd}

\usepackage{cancel}

\usepackage{verbatim}

\theoremstyle{plain}
\makeatletter
\newtheorem*{rep@theorem}{\rep@title}
\newcommand{\newreptheorem}[2]{%
\newenvironment{rep#1}[1]{%
 \def\rep@title{#2 \ref{##1}}%
 \begin{rep@theorem}}%
 {\end{rep@theorem}}}
\makeatother

\newtheorem{thm}{Theorem}[section]
\newtheorem{lem}[thm]{Lemma}
\newtheorem{prop}[thm]{Proposition}
\newtheorem{cor}[thm]{Corollary}

\theoremstyle{definition}
\makeatletter
\newtheorem*{cont@example}{\cont@title}
\newcommand{\newcontexample}[2]{%
\newenvironment{cont#1}[1]{%
 \def\cont@title{#2 \ref{##1} continued}%
 \begin{cont@example}}%
 {\end{cont@example}}}
\makeatother

\newtheorem{defn}[thm]{Definition}

\newtheorem{exmp}[thm]{Example}

\newtheorem{rem}[thm]{Remark}

\theoremstyle{remark}

\newreptheorem{thm}{Theorem}
\newreptheorem{lem}{Lemma}
\newcontexample{exmp}{Example}
\newreptheorem{prop}{Proposition}

\newcommand{\smallboxtimes}{ \mathchoice
{\mbox{\scriptsize$\boxtimes$}}
{\mbox{\scriptsize$\boxtimes$}}
{\mbox{\tiny$\boxtimes$}}
{\mbox{\tiny$\boxtimes$}}
}

\title{Curved $\mathbf{A}_\infty$-categories\footnote{We consider only $\mathbf{A}_\infty$-algebras here, however this is only to reduce the notational complexity.  All results hold for $\mathbf{A}_\infty$-categories.}  \ :  adjunction and homotopy
}
\author{Jeffrey Armstrong and Patrick Clarke}
\date{\today}

\begin{document}

\maketitle

\begin{abstract}
We develop a theory of curved $\mathbf{A}_\infty$-categories around equivalences of their module categories.  This allows for a uniform treatment of curved and 
uncurved $\mathbf{A}_\infty$-categories which generalizes the classical theory of uncurved  $\mathbf{A}_\infty$-algebras.  Furthermore, the theory is sufficiently general to treat both Fukaya categories and categories of matrix factorizations, as well as to provide a context in which unitification and categorification of pre-categories can be carried out.
 
Our theory is built around two functors: the adjoint algebra functor $U_e$
and the  functor $Q_*$.  The bulk of the paper is dedicated to proving crucial adjunction and homotopy theorems about these functors.
In addition, we explore the non-vanishing of the module categories and give a precise statement and proof the folk result known as ``Positselski-Kontsevich vanishing''.

\end{abstract}

\section{Introduction.}

The theory here treats {curved and uncurved} $\mathbf{A}_\infty$-algebras on equal 
footing and all our constructions 
are valid over an {arbitrary} unital commutative {ring} $S$.  At the heart, we 
have found a {notion of} homotopy {equivalence} for curved strictly unital $\mathbf{A}_\infty$-algebras. 
This notion {agrees with the classical notion} of $\mathbf{A}_\infty$-homotopy equivalence when
the algebras are uncurved and $S$ is a field (Theorem \ref{theorem:quillen=classical}).

Our approach is to {focus on} the category of {modules} $\operatorname{Mod_\infty}(A)$  
rather than on the $\mathbf{A}_\infty$-algebra $A$ itself.  We treat $H^0(\operatorname{Mod_\infty}(A))$
as the object that should be preserved.  Thus a morphism $A \to A'$ is defined to be an equivalence 
if it results in an equivalence of the categories $H^0(\operatorname{Mod_\infty}(A))$ and $H^0(\operatorname{Mod_\infty}(A'))$.

In order to approach the question of when  categories of the form $H^0(\operatorname{Mod_\infty}(A))$  
are equivalent, we develop {homotopy theory} which behaves functorialy with respect to morphisms $A \to A'$.
To do this we restrict attention to the subcategory of {strict morphisms} 
$\operatorname{Mod_\infty^{\text{st}}}(A)$  (defined in Corollary \ref{corollary-definition-of-mod-st-A})  of a curved strictly unital $\mathbf{A}_\infty$-algebra 
(Definition \ref{definition-of-S-linear-curved-unital-A-infinity-algebra-morphism}). There is then a functor 
$$
\operatorname{Mod_\infty^{\text{st}}} \colon \mathbf{Alg}_{\infty} \to \mathbf{ModCat}
$$
from the category 
of strictly unital curved $\mathbf{A}_\infty$-algebras to the category of Quillen model categories and Quillen adjunctions.  This functor comes with an equivalence 
$$
\operatorname{Mod_\infty^{\text{st}}}(A)[\mathcal{W}_A^{-1}] \xrightarrow{\sim}  H^0(\operatorname{Mod_\infty}(A))
$$
from the homotopy category of the strict morphisms 
to the $0^{\text{th}}$-cohomology category of  $\operatorname{Mod_\infty}(A).$  This means that a morphism $A \to A'$ of $\mathbf{A}_\infty$-algebras produces 
an adjunction 
$$
\begin{tikzcd}
H^0(\operatorname{Mod_\infty}(A)) \arrow[bend left = 10]{r}{L}
& 
H^0(\operatorname{Mod_\infty}(A')) \arrow[bend left = 10]{l}{R},
\end{tikzcd}
$$
and the adjunction of a composition is the composition of the adjunctions.

Our focus is primarily on two  amazing constructions that  make the 
homotopy theory above possible.  The first is in {Section \ref{section-adjoint-algebra}} which treats  an auxiliary curved dg-algebra $U_e(A)$ (defined in Proposition \ref{proposition-defining-UeA}) called the {adjoint algebra} whose usefulness begins with an equivalence\footnote{This is a rare example of an actual isomorphism of categories.} of categories
 $$
 \operatorname{Mod_\infty^{\text{st}}}(A) \leftrightarrow  \operatorname{Mod}_\text{dg}^\text{st}(U_e(A))
 $$
 (Theorem \ref{theorem-isomorphism-of-A-and-UeA-modules}).  
 In addition, the assignment $A \mapsto U_e(A)$ is functorial, 
and 
$U_e$ is left adjoint to the inclusion $\mathbf{Alg}_\text{dg}^\ast \hookrightarrow \mathbf{Alg}_\infty$.
In other words, $U_e(A)$ is universal in the sense 
that there is a natural isomorphism
$$
(\omega \circ -   \circ i_\bullet) \colon 
\mathbf{Alg}_\text{dg}^\ast(U_e(-), -)
 \stackrel{\sim}{\Longrightarrow}
\mathbf{Alg}_\infty(-, -)
$$
between functors on 
 $\mathbf{Alg}_\infty^\text{op} \times  \mathbf{Alg}_\text{dg}^\ast$ 
 (Theorem \ref{theorem-universal-property-of-U_e(A)}).  Ultimately, $U_e(A)$ {guarantees} the {functoriality} of our constructions.

 The second construction appears in {Section \ref{section:module-adjunctions}} is the {basis} of the Quillen {model} category {structure} on  $\operatorname{Mod_\infty^{\text{st}}}(A)$.
It begins with the fact that  $U_e(A)$ can be given the structure of an $A-A$-bimodule (Lemma \ref{lemma-UA-bimod}). This defines the functor  
 $$
 Q_A = - \overset{\infty}{\otimes} U_e(A) \colon    \operatorname{Mod}_\infty(A)  
\to
  \operatorname{Mod}_\infty(A)
 $$
 (defined in Proposition \ref{proposition-infinity-tensor}) and leads to 
 a crucially important dg-adjunction (Theorem \ref{theorem-definition-of-Q})
 $$
 \operatorname{Mod}_\text{dg}(U_e(A))(Q_A(-), -)  \stackrel{\sim}{\Longrightarrow}  \operatorname{Mod}_\infty(A)(-, -).
$$
Finally, $Q_A$ and the identity functor on  $\operatorname{Mod}_\infty(A)$
are  quasi-equivalent functors on $\operatorname{Mod}_\infty(A)$ (i.e. naturally isomorphic on $H^0(\operatorname{Mod}_\infty(A))$), and 
we show this by constructing an explicit homotopy bounded by their difference (Theorem \ref{theorem-Q-1-quasi-equivalence}).

To see how {a model structure can be built around $Q_A$}, consider an alternative viewpoint on the bar resolution.
Given a dga $\mathcal{A}$ there is an inclusion of dg-categories 
$$
i: \operatorname{Mod_\text{dg}}(\mathcal{A}) \hookrightarrow \operatorname{Mod}_\infty(\mathcal{A})
$$
and the bar resolution is left-adjoint to this inclusion:
$$
\operatorname{Hom_\text{dg}}(\text{Bar}(M), N) = \operatorname{Hom_\infty}(M,N).  
$$
Our $Q_A$ plays a similar role, but now $A$ is any curved $\mathbf{A}_\infty$-algebra and the inclusion is 
$$
i: \operatorname{Mod_\text{dg}}(U_e(A)) \hookrightarrow \operatorname{Mod}_\infty(A).
$$

The last two sections address some lingering questions.  It is in Section \ref{section:classical-theory} that we prove this new notion of equivalence agrees with the classical one  (Theorem \ref{theorem:quillen=classical}).  {Section \ref{section:non-vanishing}} considers {issues surrounding the vanishing} of the category $H^0(\operatorname{Mod_\infty}(A))$, such as those raised by 
Keller-Lowen-Nicolas \cite{keller-lowen-nicolas} and  Positselski \cite{positselski}.   This is crucial to verify that the theory here is not vacuous.  Specifically, Proposition \ref{prop-sufficient-nonvanishing} gives a criterion for non-vanishing using base-change, Theorem \ref{theorem-sufficient-vanishing} spells out explicitly when the {Kontsevich-Positselski} vanishing argument can (and cannot) be applied, and Theorem \ref{theorem-maurer-cartan-identity-image} gives an Orlov-type result \cite{orlov} which focus attention on the {critical values} of the {Maurer-Cartan} function.

Two appendices appear
after the main body of the paper.    Appendix \ref{section-some-homological-algebra} has some needed homological algebra.  The most interesting results here  are the {homotopy inversion theorem for $\mathbf{A}_\infty$-modules} (Theorem \ref{theorem:inversion-for-modules}) and the fact that if
a {morphism}  of differential graded algebras  is a {homotopy equivalence} then the associated adjunction morphisms are too (Theorem \ref{theorem:homotopic-dgs-have-homotopic-modules}).
Appendix \ref{section-homotopy-and-consequences} provides the {details of} the 
{homotopical treatment} of 
$H^0(\operatorname{Mod_\infty}(A))$.  In addition to listing explicitly all relevant details, there is the notable 
 fact that when the $\mathbf{A}_\infty$-algebra is a curved dg-algebra, our homotopy is ``compatible'' with the ``usual'' one in the appropriate sense.  This is spelled out and  we refer the reader to Armstrong \cite{armstrong} for more details.

   Several relevant aspects of the theory of curved $\mathbf{A}_\infty$-algebras are not address, or only mentioned in passing.  For instance, all our constructions are compatible with {deformations by weak bounding cochains} \cite{fooo}.  Making sure that our formulas ``commute'' with such deformations  was a condition we imposed on ourselves from the beginning of the project.  Furthermore, the theory  has all the features needed to treat both Fukaya categories \cite{fooo} and categories of matrix factorizations \cite{eisenbud}.  In fact, the apparent {paradox} that there are non-zero categories of matrix factorizations when it seems one can apply the Kontsevich-Postiselski vanishing argument is resolved here (see Example \ref{example-vanishing-matrix-factorizations}).  Another feature not mentioned is that our theory provides a framework within which one can compare (in an algebraically satisfying way) the many definitions of the Fukaya category.  For example, the natural way to attempt to turn an arbitrary curved $\mathbf{A}_\infty$-precategory $F$ \cite{kontsevich-soibelman-01} into a curved strictly unital $\mathbf{A}_\infty$-category is by looking for 
   a homotopy-initial category among categories $A$ under $F$ where the functor $F \to A$ is an inclusion.  The {FOOO strict unitification procedure} \cite{fooo} can be shown to be an example of this.

Finally, there remain vast unexplored areas.  For instance, it is natural to attempt to develop a transfer theory analogous to that  of Kadeishvili \cite{kadeishvili-80} in the uncurved case (see also Merkulov \cite{merkulov-99}), and a Hochschild cohomology governing deformations in a manner similar to that  described in Gerstenhaber \cite{gerstenhaber-63} (see also Penkava-Schwarz \cite{penkava-schwarz}).  Given our point of view, deformations would presumably run along the lines of Lowen-Van den Bergh \cite{lowen-van-den-bergh-04}.  Also, further research is required to see how Positselski's Koszul duality for weakly curved algebras \cite{positselski-weak} and the application to discrete Morse theory by Nikolov-Zahariev
\cite{nikolov-zahariev-13}
fit in with our picture.

\subsection*{Acknowledgements}
Many of the constructions here were inspired by  similar ones  already  in the literature.   
 Most notably, 
 Kenji Lef\`evre-Hasegawa used the technique of considering the subcategory of strict morphisms and a dg-algebra similar to $U_e(A)$\footnote{Note however,  $U_e(A)$ is not ``Bar-Cobar.''}
\cite{lefevre-hasegawa}. 
 
 We thank Justin Smith for clarifying  subtleties in the bar resolution, and Maxim Kontsevich who pointed out functor $- \overset{\infty}{\otimes} \mathcal{A}$ can be used to relate $\infty$-modules and ordinary modules for an uncurved dga $\mathcal{A}$ during the second author's  2013 visit
to the IH\'ES.  This visit was under the support of National Science Foundation Grant No. 1002477.

In addition, we thank Jonathan Block and 
Tony Pantev 
who allowed us to discuss this material at various points in its development in their seminars at the University of Pennsylvania, and Ludmil Katzarkov, 
Matthew Ballard 
and 
David Favero 
who invited us to speak at the 2013 Matrix Factorizations workshop in Vienna.
Finally, we thank Jim Stasheff who invited us to speak several times in his seminar and provided invaluable feedback.



\section{Notational conventions.} 
\subsection{$S$-modules, tensors and signs.}
We
fix a commutative ring $S$ with identity $1,$ and our constructions are in terms of  
$\Gamma$-graded $S$-modules for $\Gamma =  \mathbf{Z}$ or $\mathbf{Z}/2n\mathbf{Z}.$ 
 Apart from a compact notation for certain ``geometric series'' of modules and morphisms (Definition \ref{definition-geometric-series}), we follow the usual conventions for graded modules over a ring.  
The most noteworthy point among these  is that tensors of morphisms are understood to be taken with Koszul signs.

\begin{defn} {\bf (modules)} 
Usual conventions for modules include:
\begin{itemize}
\item as a module, $S$ is concentrated in degree $0$, 
\item the degree of a homogeneous element $m$ is written $|m| \in \Gamma,$ 
\item  a homogeneous element $m$ has a sign
$$
(-1)^{|m|} \in \{ \pm 1 \}
$$
(when it is unambiguous, simply $(-1)^m$\,), 
\item a shift functor 
$$
M[1]^{i} = M^{i+1},
$$
\item and an $S$-module isomorphism 
 $$
\sigma \colon M \to M[1]
$$
identifying $M^{i}$ with $M[1]^{i-1}$ and 
whose inverse is denoted $\omega.$
\end{itemize}
\end{defn}

\begin{defn} {\bf (morphisms)} Morphisms are graded
$$
\operatorname{Hom}^\Gamma_S(M, N) 
= \bigoplus_{j \in \Gamma} \operatorname{Hom}^j_S(M, N) 
= \bigoplus_{j \in \Gamma} \prod_{i \in \Gamma}  \operatorname{Hom}_S(M^i, N^{i+j}).
$$
\end{defn}

\begin{defn}{\bf (tensors)}
As usual, we have the 
graded $S$-module $M \otimes N$ for any 
graded $S$-modules $M$ and $N$.  Morphisms  can be tensored:  for $\phi \colon M \to N$ and $\psi \colon M' \to N',$
we have 
$$
\phi \otimes \psi \colon M \otimes M' \to N \otimes N'.
$$
This map is taken with Koszul signs;  that is to say,  for a tensor of homogeneous morphisms evaluated on a tensor of homogeneous elements, we have
$$
(\phi \otimes \psi)(m \otimes m') = (-1)^{|\psi| |m|} \phi(m) \otimes \psi(m').
$$
\end{defn}

\begin{defn}{\bf (geometric series)}
\label{definition-geometric-series}
We make the notation 
$$
M^\otimes = \bigoplus_{\ell = 0}^\infty M^{\otimes \ell},
$$
for the {\bf tensor space}
where $M^{\otimes \ell} = \underbrace{M \otimes \dotsm \otimes M}_{\text{$\ell$ copies}}$
and $M^{\otimes 0} = S \cdot 1_\otimes$  for a symbol $1_\otimes $. We also write $M^{\otimes \geq 1}$ for $\bigoplus_{\ell = 1}^\infty M^{\otimes \ell}$.

A similar notation is used for morphisms.  Given $\phi_\bullet \colon M^\otimes \to N,$ we write the {\bf geometric series} as
$$
(\phi_\bullet)^\otimes = \sum_{\ell = 0}^\infty (\phi_\bullet)^{\otimes \ell} \colon M^\otimes \to N^\otimes
$$
where 
$$(\phi_\bullet)^{\otimes \ell} = \sum_{i_1, \dotsm, i_\ell} \phi_{i_1} \otimes \dotsm \otimes \phi_{i_\ell}.$$
We also make use of the ``extension by zero'' convention for direct sums: for a map $\psi$ on a summand of a direct sum, we extend it to the entire direct sum by first  projecting onto the summand in question.\end{defn}

\begin{exmp} {\bf (identity map)}
Using the geometric series conventions (Definition \ref{definition-geometric-series})
the identity map $M^\otimes \to M^\otimes$ can be written 
$\mathbf{1}^{{\otimes}}$ where $\mathbf{1}$ is the identity map $M \to M.$ The domain extension convention allows $\mathbf{1}$ to be thought of as a map $M^\otimes \to M$.
\end{exmp}

\begin{rem} {\bf(tensors of tensors)}
At times we will need to take tensors of tensors, etc.  Such an object is multi-graded by the resulting tensor degrees.  To reflect this structure, we will sometimes use the  notations ${\smallboxtimes}, \odot, \diamond,$ and $\oslash$ for  additional tensors.  

All formulas involving tensors should be understood to be relative to some base tensor symbol in the formula.  This is because the practice of taking tensors of tensors, while seemingly innocuous, is subtle.  For example, 
there is a map 
$$
V \smallboxtimes W \to V \otimes W 
$$
which switches the $\smallboxtimes$ to $\otimes$.  This  is usually an isomorphism.  However if the symbols are not interpreted in a relative way, then it may not be.
Indeed, for $U \neq 0$ and   $V = W = U^\otimes$ this map has a kernel.  
\end{rem}

\begin{rem}
{\bf (signs) } In the interest of clarity, we have left most sign checking out of our proofs.
The job of keeping track of signs can be quite formidable, and can obscure the content of a proof. 
For those interested in checking signs, we recommend drawing  string diagrams between
steps in a computation and 
  counting crossings\footnote{We first saw this trick in Dan Abramovich's notes from Maxim Kontsevich's deformation theory course.}. 
\end{rem}

\begin{figure}[h!]
\centering
\begin{tikzpicture}


\foreach \x/\y in {1/$($,1.5/$\sigma$,2/$\odot$,2.5/$1^{\smallboxtimes}$,3/$\odot$,3.5/$\sigma$,4/$)($, 4.5/$D_2$, 5/${\smallboxtimes}$, 5.5/$1^{\smallboxtimes}$, 6/$)($, 6.5/$\omega$, 7/${\smallboxtimes}$, 7.5/$1^{\smallboxtimes}$, 8/${\smallboxtimes}$, 8.5/$\omega$, 9/$)$}
\node (a\x) at (\x, 2) {\y};


\foreach\v/\w in {2/$\big($, 2.5/$($,3/$\sigma$ , 3.5/$\odot$, 4/$1^{\smallboxtimes}$, 4.5/$)$, 5/$D_2$, 5.5/$\omega$, 6/$\big)$,6.5/${\smallboxtimes}$, 7/$1^{\smallboxtimes}$, 7.5/$\odot$, 8/$1^\otimes$}
\node (b\v) at (\v, 0) {\w};


\draw[->] { [rounded corners] (1.5,1.75) .. controls (2, 1) .. (b3)} [bend left] {};
\draw[->] (3.5,1.75) .. controls (5,1) and (7,1) .. (b8) {};
\draw[->]{ (4.5,1.75) .. controls (5,1) .. (b5)} {};
\draw[->]{  (6.5,1.75) .. controls (6,1) ..  (5.5, .25) } {};
\draw[->](8.5,1.75) -- (b8) {};


\node at (4.8,1.3) [shape= circle, draw] {};
\node at (6,1) [shape= circle, draw] {};

\end{tikzpicture}
\caption{Checking a sign from the proof of Lemma \ref{lemma-UA-bimod} by counting crossings.}
\end{figure}
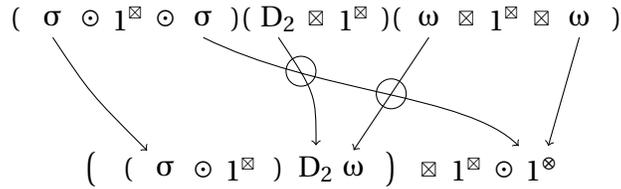


\subsection{Strictly unital $\mathbf{A}_\infty$-algebras.}

We define and verify the category $\mathbf{Alg}_\infty$ of strictly unital $\mathbf{A}_\infty$-algebras. 
This allows us to establish some notational conventions and touch upon typical arguments made when 
 doing $\mathbf{A}_\infty$ computations: the cross-cancel argument and the codifferential property.  We primarily use
 the $b$-notation, and capital letters for the operators on coalgebras.  The more traditional $\frak{m}$ notation
 is introduced in Subsection \ref{subsection:bar-construction}.

\begin{defn}{\bf (strictly unital $\mathbf{A}_\infty$-algebra)}
\label{definition-of-S-linear-curved-unital-A-infinity-algebra}
An $\mathbf{A}_\infty$-algebra is an 
$S$-module $A$ and  a  map
$$
b_\bullet = \sum_{\ell = 0}^\infty b_\ell \in \operatorname{Hom}^1_S((A[1])^\otimes, A[1]).
$$
such that 
$$
B^2 = 0
$$
for $B = \mathbf{1}^\otimes \otimes b_\bullet \otimes \mathbf{1}^\otimes.$  An element 
$e \in A^0$ is called a strict unit if sequence of $S$-modules
$$
0 \to S \to  A \to A/ e \to 0
$$
is split exact, and
for $\eta = \sigma(e)$
\begin{itemize}
\item $ b_2(\eta \otimes \mathbf{1})  = -b_2(\mathbf{1} \otimes \eta)   = \mathbf{1}$, and 
\item $b_\ell(\mathbf{1}^{\otimes } \otimes \eta \otimes \mathbf{1}^{\otimes }) = 0$ for $\ell \neq 2.$
\end{itemize}
Such an element is unique if it exists.
\end{defn}

\begin{prop} {\bf ($\mathbf{A}_\infty$-equations on $b_\bullet$)}
\label{A-infinity-equations-on-b}
The condition
$$
B^2 = 0
$$
is equivalent to 
$$
b_\bullet(B) = b_\bullet(\mathbf{1}^\otimes \otimes b_\bullet \otimes \mathbf{1}^\otimes) = 0.
$$
\end{prop}
\begin{proof}{\bf (cross-cancel argument)}
The argument here is typical for these types of statements.  The composition
$B^2$ expands into a sum of three terms.  Two of these terms are identical, but with opposite signs
which arise from whether or not the two $b_\bullet$ crossed each other.
The remaining term  is a map sandwiched between two $\mathbf{1}^\otimes$'s, so the result is zero if
and only if the sandwiched map is zero.  Explicitly
$$
\begin{array}{rcl}
B^2 & = & (\mathbf{1}^\otimes \otimes b_\bullet \otimes \mathbf{1}^\otimes) (\mathbf{1}^\otimes \otimes b_\bullet \otimes \mathbf{1}^\otimes) \\
& = & \mathbf{1}^\otimes \otimes b_\bullet(\mathbf{1}^\otimes \otimes b_\bullet \otimes \mathbf{1}^\otimes)  \otimes \mathbf{1}^\otimes\\ 
&  & + \mathbf{1}^\otimes \otimes b_\bullet \otimes \mathbf{1}^\otimes \otimes b_\bullet \otimes \mathbf{1}^\otimes \text{ \qquad (did not cross) }\\
& & - \mathbf{1}^\otimes \otimes b_\bullet \otimes \mathbf{1}^\otimes \otimes b_\bullet \otimes \mathbf{1}^\otimes \text{ \qquad (crossed) } \\
& = &  \mathbf{1}^\otimes \otimes b_\bullet(\mathbf{1}^\otimes \otimes b_\bullet \otimes \mathbf{1}^\otimes)  \otimes \mathbf{1}^\otimes.
\end{array}
$$
\end{proof}

\begin{rem} {\bf (comultiplication and codifferential)}
In the language of coalgebras, the operator $B$ is a codifferential on $(A[1])^\otimes$
considered as a coalgebra with usual tensor comultiplication 
$$
\Delta \colon  (A[1])^\otimes \to  (A[1])^\otimes \smallboxtimes (A[1])^\otimes
$$
which sends
$$
x_1 \otimes \dotsm \otimes x_k \mapsto 
\sum_{i = 0}^{k} 
x_1 \otimes \dotsm \otimes x_i \smallboxtimes
x_{i+1} \otimes \dotsm \otimes x_k.
$$
This means that in addition to $B^2=0$, we have
$$
\Delta B = (B \smallboxtimes \mathbf{1}^\otimes + \mathbf{1}^\otimes \smallboxtimes B) \Delta.
$$
\end{rem}

\begin{defn} {\bf (morphism of strictly unital $\mathbf{A}_\infty$-algebras)}
\label{definition-of-S-linear-curved-unital-A-infinity-algebra-morphism}
A morphism of strictly unital $\mathbf{A}_\infty$-algebras $f_\bullet \colon (A, b_\bullet) \to (A', b'_\bullet)$ is an element
$$f_\bullet = \sum_{\ell = 1}^\infty f_\ell  \ \in  \ \operatorname{Hom}^0_S((A[1])^{\otimes \geq 1}, A'[1])$$ such that
\begin{itemize}
\item $0 = B' F - F B \colon (A[1])^\otimes \to (A'[1])^\otimes$ for $F = (f_\bullet)^\otimes$, 
\item $f_1(\eta) = \eta'$, and
\item $f_\ell(\mathbf{1}^{\otimes } \otimes \eta \otimes \mathbf{1}^{\otimes }) = 0$ for  $\ell \neq 1$.
\end{itemize}
\end{defn}

\begin{prop}{\bf ($\mathbf{Alg}_\infty$)} 
\label{proposition-composition-of-algebra-morphisms}  Strictly unital $\mathbf{A}_\infty$-algebras form a category $\mathbf{Alg}_\infty$ with composition
$$
g_\bullet \circ f_\bullet = g_\bullet((f_\bullet)^\otimes) \in \operatorname{Hom}^0_S((A[1])^{\otimes \geq 1}, A''[1])
$$
for morphisms $f_\bullet \colon (A, b_\bullet) \to (A', b'_\bullet)$ and 
$g_\bullet \colon (A', b'_\bullet) \to (A'', b''_\bullet)$.  If one sets $h_\bullet = g_\bullet \circ f_\bullet$,
then 
 $$
 H = GF
 $$
 for $F = (f_\bullet)^\otimes$,
  $G = (g_\bullet)^\otimes$, and
   $H = (h_\bullet)^\otimes$, where concatenation indicates the usual composition of maps.
\end{prop}
\begin{proof}
{\bf (composition of geometric series)} The identity $H = GF$ follows from the fact that $(g_\bullet((f_\bullet)^\otimes))^{\otimes} =  (g_\bullet)^{\otimes}  (f_\bullet)^\otimes$.  This equality reflects the term-wise equality
$$
(g_{i_1} \otimes \dotsm \otimes g_{i_k}) (f_{j_1^{1}} \otimes \dotsm \otimes f_{j^1_{i_1}}) \otimes \dotsm \otimes (f_{j_1^{k}} \otimes \dotsm \otimes f_{j^k_{i_k}}) =
g_{i_1}(f_{j_1^{1}} \otimes \dotsm \otimes f_{j^1_{i_1}}) \otimes \dotsm \otimes 
g_{i_k}(f_{j_1^{k}} \otimes \dotsm \otimes f_{j^k_{i_k}})
$$
which holds when all maps are even (and with appropriate signs when they are not). 

The condition that $H$ ``commutes'' with the $B$'s follows:
$$
B''H = B''GF = GB'F = GFB = HB.
$$
 
  Finally, on strict units: 
$$
h_\ell( \mathbf{1}^\otimes \otimes \eta \otimes \mathbf{1}^\otimes) = f_\bullet( \ (g_\bullet)^\otimes \otimes \eta' \otimes (g_\bullet)^\otimes \ )|_{(A[1])^{\otimes \ell}} = 
\left\{
\begin{array}{ll}
f_1(g_1(\eta)) = \eta'' &\text{if } \ell = 0, \\
0 & \text{otherwise }
\end{array}
\right.
$$ 
as needed.
\end{proof}


\subsection{Strictly unital right modules.}

Strictly unital right modules over an $\mathbf{A}_\infty$-algebra $A$ form a dg-category $\operatorname{Mod}_\infty(A)$.
Within this category is a subcategory of strict morphisms
$\operatorname{Mod_\infty^{\text{st}}}(A).$ The strict morphism category plays a central role later when it is identified with the module category of the curved dg-algebra $U_e(A)$ (Theorem \ref{theorem-isomorphism-of-A-and-UeA-modules}).

\begin{defn}{\bf (strictly unital right module)}
\label{definition-of-unital-right-module}
A strictly unital right module over a strictly unital $\mathbf{A}_\infty$-algebra is an $S$-module $M$ and a 
map 
$$
b^M_\bullet = \sum_{\ell = 0}^\infty b^M _\ell \in \operatorname{Hom}^0_S(M \odot (A[1])^\otimes, M)
$$
such that 
$$
(B^M)^2 = 0
$$
for $B^M = b_\bullet^M \odot \mathbf{1}^\otimes + \mathbf{1} \odot \mathbf{1}^\otimes \otimes b_\bullet \otimes \mathbf{1}^\otimes$.  In addition, 
\begin{itemize}
\item $-b_2^M(\mathbf{1} \odot \eta)   = \mathbf{1}$, and 
\item $b^M_\ell(\mathbf{1} \odot \mathbf{1}^{\otimes } \otimes \eta \otimes \mathbf{1}^{\otimes }) = 0$ for $\ell \neq 2.$
\end{itemize}
\end{defn}

\begin{prop} {\bf ($\mathbf{A}_\infty$-module equations on $b^M_\bullet$)}
The condition
$$
(B^M)^2 = 0
$$
is equivalent to 
$$
b^M_\bullet(B^M)  = b^M_\bullet(b_\bullet^M \odot \mathbf{1}^\otimes + \mathbf{1} \odot \mathbf{1}^\otimes \otimes b_\bullet \otimes \mathbf{1}^\otimes) = 0.
$$
\end{prop}
\begin{proof}
This uses the same  cross-cancel argument as the proof of Proposition \ref{A-infinity-equations-on-b}.
\end{proof}

\begin{prop} {\bf ($\operatorname{Mod}_\infty(A)$)}
\label{proposition-definition-of-mod-infinity-A}
Strictly unital right modules over an $\mathbf{A}_\infty$-algebra $A$ form a dg-category $\operatorname{Mod}_\infty(A)$ whose morphism complex is
$$
\operatorname{Mod}_\infty(A)(M, M') = \operatorname{Hom}_S(M \odot (A[1])^\otimes, M')
$$
with the differential 
$$
\delta \phi_\bullet =   b^{M''}_\bullet(\phi_\bullet \odot \mathbf{1}^\otimes) - (-1)^{\phi_\bullet} \phi_\bullet(b_\bullet^{M'} \odot \mathbf{1}^\otimes + \mathbf{1} \odot \mathbf{1}^\otimes \otimes b_\bullet \otimes \mathbf{1}^\otimes) 
$$
and composition 
 $$\psi_\bullet \circ \phi_\bullet = \psi_\bullet(\phi_\bullet \odot \mathbf{1}^\otimes) \in \operatorname{Hom}_S(M \odot (A[1])^\otimes, M'')$$
 for morphisms $\phi_\bullet \colon (M, b^M_\bullet) \to (M', b^{M'}_\bullet)$, and
  $\psi_\bullet \colon (M', b^{M'}_\bullet) \to (M'', b^{M''}_\bullet).$
 
 Furthermore, if one sets
 $\theta_\bullet = \psi_\bullet \circ \phi_\bullet$ then
 these maps satisfy
 $$
\Theta  = \Psi \Phi
 $$
 for $\Phi = \phi_\bullet \odot \mathbf{1}^\otimes$,
  $\Psi = \psi_\bullet \odot \mathbf{1}^\otimes$, and
   $\Theta = \theta_\bullet \odot \mathbf{1}^\otimes$.  In addition, the operator 
   $$[B, \Phi] =
   B^{M'} \Phi - (-1)^{\Phi} \Phi B^{M}$$ satisfies
$$
   [B,\Phi] = (\delta\phi_\bullet) \odot \mathbf{1}^\otimes.
$$
\end{prop}
\begin{proof} 
These statements are equivalent to the statement that 
$$
\phi_\bullet \mapsto \Phi
$$
includes $\operatorname{Hom}_S(M \odot (A[1])^\otimes, M')$ as a subcomplex of 
$\operatorname{Hom}_S(M \odot (A[1])^\otimes, M' \otimes (A[1])^\otimes)$ which 
is closed under composition.   
The equality
$$
[B,\Phi] = (\delta \phi_\bullet) \odot \mathbf{1}^\otimes
$$
guarantees the ``subcomplex'' statement, and can be verified using the cross-cancel argument of the proof of Proposition \ref{A-infinity-equations-on-b}.  The ``composition'' statement is checked by the same term-wise equality that appeared in the composition of geometric series in the proof of Proposition \ref{proposition-composition-of-algebra-morphisms}.
\end{proof}

\begin{defn}{\bf (strict morphisms)}
\label{definition-of-strict-morphism}
A strict morphism of strictly unital right modules over an $\mathbf{A}_\infty$-algebra  $(M, b^M_\bullet) \to (M', b^{M'}_\bullet)$ is a closed degree $0$ morphism whose only non-zero element is ``linear'':
$$\phi_\bullet = \phi_1 \in \operatorname{Hom}^0_S(M, M') \cap Z^0(\operatorname{Mod}_\infty(M,M')).$$
\end{defn}

\begin{cor} {\bf ($\operatorname{Mod_\infty^{\text{st}}}(A)$)}
\label{corollary-definition-of-mod-st-A}
The composition of strict morphisms is again a strict morphism and thus define a subcategory
$
\operatorname{Mod_\infty^{\text{st}}}(A) \subseteq \operatorname{Mod_\infty}(A).
$ 
\end{cor}
\begin{proof}
Omitted.
\end{proof}


\subsection{The $\frak{m}$'s.}
\label{subsection:bar-construction}
The definitions in terms of $b_\bullet$ are equivalent to  a commonly used formulation of $\mathbf{A}_\infty$-algebras
in terms of a map
$$
\frak{m}_\bullet \colon A^\otimes \to A
$$
where $|\frak{m}_i| = i-2.$  Both descriptions appeared in Stasheff's original work on the subject
 \cite{stasheff-published-thesis}. The equivalence related to the bar resolution of   Eilenberg-Mac Lane \cite{eilenberg-maclane-bar-construction}.  
 For us, the main reason to use the formulation in terms of $\frak{m}_\bullet$ is that we will consider curved dg-algebras and their modules, and these objects are neatly described in this language.
 The essential fact is the following.
 \begin{prop} {\bf ($\frak{m}$-$b$-equivalence) }
 \label{proposition-bar-construction}
The assignment
$$
\frak{m}_i \mapsto -\sigma \circ  \frak{m}_i  \circ \omega^{\otimes i} = b_i
$$
identifies those $m_\bullet \in \operatorname{Hom}_S(A^\otimes, A)$ which satisfy
$$
0 = \sum_{j, k, \ell} (-1)^{jk+\ell} \frak{m}_i(\mathbf{1}^{\otimes j} \otimes \frak{m}_k \otimes \mathbf{1}^{\otimes \ell})
$$
with those $b_\bullet \in \operatorname{Hom}_S(A[1]^\otimes, A[1])$ which satisfy
$$
0 = b_\bullet(\mathbf{1}^\otimes \otimes b_\bullet \otimes \mathbf{1}^\otimes).
$$
\end{prop}
\begin{proof}
Omitted.
\end{proof}


\subsection{Curved dg-algebras and their modules.}

The relevance of curved dg-algebras to this paper
comes from the adjoint algebra $U_e(A)$ (Proposition \ref{proposition-defining-UeA}).
There is a slight difference in what is meant by a \emph{morphism} of curved dg-algebras
depending on whether the are considered for their own sake or as special cases of $\mathbf{A}_\infty$-algebras (those with $b_{\geq 3} = 0$).
We consider them as special cases of $\mathbf{A}_\infty$-algebras, and thus get a category we denote by 
$\mathbf{Alg}_\text{dg}^\ast$.  The alternative view produces a category  $\mathbf{Alg}_\text{dg}$
and $\mathbf{Alg}_\text{dg}^\ast$ is a subcategory.  

A morphism in 
$\mathbf{Alg}_\text{dg}$
 includes the data of an element $\mathfrc{b}' \in (\mathcal{A}')^1$ and
modifies the last two equations of Proposition \ref{prop:alg_dg-ast} in a way equivalent to allowing an $f_0$ term  in the $\mathbf{A}_\infty$-morphism \cite{positselski}.  This larger category does not include into 
 $\mathbf{Alg}_\infty$;  the problem being that composition with an $f_0$  term  can lead to an infinite sum for which convergence addressed.  In the geometric picture of 
 Kontsevich-Soibelman \cite{kontsevich-soibelman}, the  morphisms in $\mathbf{Alg}_\text{dg}^\ast$
 are the basepoint preserving ones.

\begin{defn} {\bf (curved dg-algebra)}  
A curved dg-algebra $(\mathcal{A}, d, c)$ is an $\mathbf{A}_\infty$-algebra
$$
\frak{m}_\bullet \colon \mathcal{A}^\diamond \to \mathcal{A}
$$
with $\frak{m}_\ell = 0$ for all $\ell \geq 3.$  Notationally, we write
\begin{itemize}
\item multiplication by  $a \cdot a' = \frak{m}_2(a \diamond a')$ (or concatenation),
\item the derivation\footnote{The derivation is called the ``differential'' despite the fact that it may not satisfy $d^2 = 0$.} by $da = \frak{m}_1(a)$, and 
\item the curvature by $c = m_0(1)$,
\end{itemize}
and the $\mathbf{A}_\infty$-equations are exactly the conditions that
\begin{itemize}
\item $dc = 0$,
\item $d^2a = [c,a] = ca - ac$,
\item $d(a \cdot a') = (da) \cdot a' + (-1)^a a \cdot (da'),$ and
\item $(a \cdot a') \cdot a'' = a \cdot (a' \cdot a'')$.
\end{itemize}
In addition, a strict unit $\eta$  defines an identity element $e  = \omega(\eta) \in \mathcal{A}$ which is closed
$$
de = 0.
$$
\end{defn}

\begin{prop}{\bf ($\mathbf{Alg}_\text{dg}^\ast$)} 
\label{prop:alg_dg-ast}
The category $\mathbf{Alg}_\text{dg}^\ast$ of strictly unital curved dg-algebras 
is made up of degree 0 morphisms of graded algebras
$$
\mathfrc{f} \colon \mathcal{A} \to \mathcal{A}'
$$
for which $\mathfrc{f}(e) = e'$, 
\begin{itemize}
\item $\mathfrc{f}(a \cdot a')  = \mathfrc{f}(a) \cdot  \mathfrc{f}(a')$,
\item $d' \mathfrc{f}(a) = \mathfrc{f}(da)$, and  
\item $\mathfrc{f}(c) = c'$ 
\end{itemize}
The identification 
$f_1  = \sigma \circ \mathfrc{f} \circ  \omega$ 
defines a
``linear'' morphism
$$f_\bullet = 
 f_1 \colon  
 \mathcal{A}[1]  \to \mathcal{A}'[1] \in \mathbf{Alg}_\infty(\mathcal{A}, \mathcal{A}')$$
for which the bulleted equations  become exactly the condition  $B'F-FB= 0$, and thus 
$\mathbf{Alg}_\text{dg}^\ast$ includes as a subcategory of $\mathbf{Alg}_\infty.$
\end{prop}
\begin{proof}
Omitted.
\end{proof}
 
\begin{prop} {\bf (dg-module over a curved dg-algebra)}
A unital right dg-module $(N, d)$ over a curved dg-algebra $\mathcal{A}$ is a  unital right $\mathbf{A}_\infty$-module 
$$
\frak{m}_\bullet^{N} \colon N \odot \mathcal{A}^\diamond \to N
$$
with  $\frak{m}^N_\ell = 0$ for all $\ell \geq 3.$ 
Notationally, we write
\begin{itemize}
\item multiplication by  $n \cdot a = \frak{m}^N_2(n \odot a)$, and
\item the derivation by $dn = \frak{m}^N_1(n)$, 
\end{itemize}
and the $\mathbf{A}_\infty$-equations are exactly the conditions that
\begin{itemize}
\item $d^2n =  n \cdot c$,
\item $d(n \cdot a) = (dn) \cdot a + (-1)^n n \cdot (da),$ and
\item $(n \cdot a) \cdot a' = n \cdot (a \cdot a')$.
\end{itemize}
\end{prop}
\begin{proof}
Omitted.
\end{proof}

\begin{defn}{\bf ($\operatorname{Mod}_\text{dg}(\mathcal{A})$)} 
The category
$\operatorname{Mod}_\text{dg}(\mathcal{A})$ 
is a dg-category 
made up of unital right dg-modules over $\mathcal{A}$ whose morphism complexes are 
$$
\operatorname{Mod}_\text{dg}(\mathcal{A})(N, N') = 
\operatorname{Hom}_{\mathcal{A}}(N, N')
$$
of $\mathcal{A}$-linear maps with the differential
$$
[d, \varphi] = d' \circ \varphi- (-1)^{\varphi} \ \varphi \circ d. 
$$
Setting $\phi_1 = -\omega \circ \varphi \circ \sigma$ defines a
``linear'' morphism
$$\phi_\bullet = \phi_1 \colon N[1]  \to N'[1] \in \operatorname{Mod_\infty}(A)(N[1])(N'[1])$$
for which 
$$
\delta \phi_\bullet = -\omega \circ [d, \varphi] \circ \sigma
$$
and thus 
$\operatorname{Mod}_\text{dg}(\mathcal{A})$ includes as a subcategory of $\operatorname{Mod}_\infty(\mathcal{A}).$
\end{defn}

\begin{defn}{\bf ($\operatorname{Mod}^{\text{st}}_\text{dg}(\mathcal{A})$)} 
The category
$\operatorname{Mod}^{\text{st}}_\text{dg}(\mathcal{A})$ of strict morphisms
is the subcategory $Z^0(\operatorname{Mod}_\text{dg}(\mathcal{A})) \subseteq \operatorname{Mod}_\text{dg}(\mathcal{A})$.
The identity $[d, \varphi] = 0$ is equivalent to the conditions that 
\begin{itemize}
\item $\varphi(m \cdot a)  = \varphi(m) \cdot  a$, and
\item $d' \varphi(m) = \varphi(dm)$.
\end{itemize}
\end{defn}


\subsection{Bimodules.}

An $A$-$A'$-bimodule $V$ for curved $\mathbf{A}_\infty$-algebras $A$ and $A'$ defines a functor 
$$
  - \overset{\infty}{\otimes} V \colon 
   \operatorname{Mod}_\infty(A)  
\to
  \operatorname{Mod}_\infty(A').
 $$
The functor $Q$ (Theorem \ref{theorem-definition-of-Q}) which produces the dg-adjunction between $A$ modules and $U_e(A)$ modules is defined this way.  Here we recall the 
definition of a bimodule and the resulting functor.

\begin{defn}{(\bf $\infty$-bimodule)}
 \label{definition-infinity-bimodule}
 An $A$-$A'$-bimodule $V$ for curved $\mathbf{A}_\infty$-algebras $A$ and $A'$
 is an $S$-module and a map
 $$
 b^V_{\bullet,  \bullet}\colon A[1]^\otimes \odot V \odot (A')[1]^\otimes \to V
 $$
 such that 
 $$
(B^V)^2 = 0
 $$
 for 
 $$
 B^V= B \odot \mathbf{1} \odot \mathbf{1}^\otimes + \mathbf{1}^\otimes \odot b^V_{\bullet, \bullet} \odot \mathbf{1}^\otimes + \mathbf{1}^\otimes \odot \mathbf{1} \odot B'.
 $$
 The $\mathbf{A}_\infty$-equation for such a map is
$$
b^V_{\bullet, \bullet}(B \odot \mathbf{1} \odot \mathbf{1}^\otimes + \mathbf{1}^\otimes \odot b^V_{\bullet, \bullet} \odot \mathbf{1}^\otimes + \mathbf{1}^\otimes \odot\mathbf{1} \odot B') = 0.
$$

If $A$ and $A'$ are strictly unital, then $V$ is called strictly unital if  
$$b^V_{\bullet,  \bullet}(\sigma(a_1) \otimes \dotsm \otimes \sigma(a_k) \odot v \odot \sigma(a'_1) \otimes \dotsm \otimes \sigma(a'_{k'})) = 0$$
when $k + 1+ k' > 2$ and any $a_i = e$ or $a'_{i'} = e'$.
\end{defn}

\begin{prop}
{\bf (infinity-tensor)}
 \label{proposition-infinity-tensor}
 An $A$-$A'$-bimodule $V$ defines a dg-functor
 $$
  - \overset{\infty}{\otimes} V \colon 
   \operatorname{Mod}_\infty(A)  
\to
  \operatorname{Mod}_\infty(A') 
 $$
 which sends
 $$
 M \mapsto M \overset{\infty}{\otimes} V  = M \odot A[1]^\otimes \odot V 
 $$
 with 
 $$
 b^{M \overset{\infty}{\otimes} V}
 \colon M \odot A[1]^\otimes \odot V \odot A'[1]^\otimes \to M \odot A[1]^\otimes \odot V
 $$
 equal to 
 $$B^M \odot\mathbf{1}  
 + \mathbf{1} \odot \mathbf{1}^\otimes \odot b^V_{\bullet, \bullet},$$
 and 
 $\phi_\bullet \colon M \to M'$ to the linear morphism 
 $$
( \phi_\bullet \odot \mathbf{1}^\otimes \odot \mathbf{1} )
 \colon M \odot A[1]^\otimes \odot V  \longrightarrow M' \odot A[1]^\otimes \odot V.
 $$
\end{prop}
\begin{proof}
The expression to check for objects simplifies using the cross-cancel argument and $B^2 = 0.$
The remaining terms can be organized by whether or not an element of $M$ passes through a $B$
map:
$$
\begin{array}{rcl}
(B^{M \overset{\infty}{\otimes} V})^2 & = & (B^M \odot \mathbf{1} \odot \mathbf{1}^\otimes  
 + \mathbf{1} \odot \mathbf{1}^\otimes \odot b^V_{\bullet, \bullet} \odot \mathbf{1}^\otimes + 
  \mathbf{1} \odot \mathbf{1}^\otimes \odot \mathbf{1} \odot B' )^2
\\
& = & b^M_\bullet(B^M) \odot 
 \mathbf{1}^\otimes \odot 
\mathbf{1} \odot \mathbf{1}^\otimes  + \mathbf{1} \odot (B^V)^2 \\
& = & 0.\\
\end{array}
$$
The last equality, uses the $\mathbf{A}_\infty$-equations for $M$:  $b^M_\bullet(B^M)  = 0$.

On morphisms, 
$$
\begin{array}{lll}
B^{M \overset{\infty}{\otimes} V} 
 \circ  & 
 \phi_\bullet \odot \mathbf{1}^\otimes \odot \mathbf{1} \odot \mathbf{1}^\otimes & 
\\
- (-1)^{\phi} & 
\phi_\bullet \odot \mathbf{1}^\otimes \odot \mathbf{1} \odot \mathbf{1}^\otimes
&  \circ 
B^{M \overset{\infty}{\otimes} V} 
\end{array}
\ \ = \ \ 
\begin{array} {l}
b^M_\bullet  (\phi_\bullet \odot \mathbf{1}^\otimes ) \odot \mathbf{1}^\otimes \odot \mathbf{1} \odot \mathbf{1}^\otimes
+  (-1)^\phi \cancel{\phi_\bullet \odot  B^V} \\
  - (-1)^\phi [ \phi_\bullet (B_M) \odot \mathbf{1}^\otimes \odot \mathbf{1} \odot \mathbf{1}^\otimes
+ \cancel{\phi_\bullet \odot B^V}
 ]
\end{array}
$$
$= (\delta\phi_\bullet) \odot  \mathbf{1}^\otimes \odot \mathbf{1} \odot \mathbf{1}^\otimes.$
\end{proof}


\section{The adjoint algebra.}
\label{section-adjoint-algebra}

The adjoint algebra $U_e(A)$ (Proposition \ref{proposition-defining-UeA}) is used largely for the same purposes as the algebra $\mathbb{U}^+(A)$ used by Lef\`evre-Hasegawa \cite{lefevre-hasegawa}.  Like $\mathbb{U}^+(A)$, the algebra $U_e(A)$ is universal in the appropriate sense (Theorem \ref{theorem-universal-property-of-U_e(A)}), and
$U_e(A)$-modules reproduce the category 
$\operatorname{Mod}^\text{st}_\infty(A)$ (Theorem \ref{theorem-isomorphism-of-A-and-UeA-modules}). 

However, $U_e(A)$ is not obtained in the usual way as a ``bar-cobar''.
Superficially it looks similar, but it is sneakily different.  It is closer to think of it as first doing a lunatic version of bar-cobar in which one uses the non-reduced tensor coalgebra equipped with the reduced comultiplication.  This is then quotiented by an ideal the guarantees things behave appropriately with the unit.


\subsection{$U_e(A)$ definitions.}

$U_e(A)$ is a quotient of an algebra $U(A)$ and most of our constructions are done on the level of $U(A).$
However, they descend to  $U_e(A)$ and thus allow access to  
 $\operatorname{Mod}^\text{st}_\infty(A)$.

\begin{defn}{\bf ($U(A)$)}
The enveloping algebra
$U(A)$ is a curved dg-algebra obtained by setting
$X = ((A[1])^{\otimes \geq 1})[-1]$ and 
$$U(A) =  X^{\smallboxtimes}.$$
It is equipped with the derivation
$$
d = d_1 + \overline{d}_2
$$
whose value on generators is $D = D_1 + \overline{D}_2$ for
\begin{itemize}
\item $D_1 = -\omega \circ B \circ \sigma,$ and
\item $\overline{D}_2 = -\omega^{{\smallboxtimes} 2} \circ \overline{\Delta} \circ \sigma$
\end{itemize}
where $\overline{\Delta}$ is the comultiplication on the reduced tensor coalgebra  $(A[1])^{\otimes \geq 1}$.
This means
$$
\overline{\Delta}(\sigma(a_1) \otimes \dotsm \otimes \sigma(a_\ell)) = 
\sum_{i = 1}^{\ell-1} \sigma(a_1) \otimes \dotsm \otimes \sigma(a_i) {\smallboxtimes}
\sigma(a_{i+1}) \otimes \dotsm \otimes \sigma(a_\ell).
$$
Equivalently, $\overline{\Delta} = \Delta - \mathbf{1}^\otimes {\smallboxtimes} 1_\otimes - 1_\otimes {\smallboxtimes}\mathbf{1}^\otimes$
were $\Delta$ is the  comultiplication on the full tensor coalgebra  $(A[1])^{\otimes}$.
\end{defn}

\begin{prop}{\bf ($U(A)$ verification)}
$(U(A), d, c)$ is a curved dg-algebra with $c = \frak{m}_0(1)$ and
 identity element $1 \in  X^{\otimes 0} \cong S$.
\end{prop}
\begin{proof}
$U(A)$ is tensor algebra, and one can always extend a map $D \colon X \to X^{\smallboxtimes}$
to a derivation
$$
d = \mathbf{1}^{\smallboxtimes} {\smallboxtimes} D {\smallboxtimes}  \mathbf{1}^{\smallboxtimes} \colon X^{\smallboxtimes} \to X^{\smallboxtimes}.
$$ 

It remains to check $d^2x = [c,x]$.  For this, consider
$$
d^2 = (\mathbf{1}^{\smallboxtimes} {\smallboxtimes} D {\smallboxtimes}  \mathbf{1}^{\smallboxtimes})^2.
$$
The cross-cancel argument of the proof of 
 Proposition \ref{A-infinity-equations-on-b} simplifies this to
 $$
 \begin{array}{rcl}
d^2 & = & \mathbf{1}^{\smallboxtimes} {\smallboxtimes} D_1^2 {\smallboxtimes}  \mathbf{1}^{\smallboxtimes} \\
& & +  \mathbf{1}^{\smallboxtimes} {\smallboxtimes} \overline{D}_2D_1 {\smallboxtimes}  \mathbf{1}^{\smallboxtimes} \\
& & + 
\mathbf{1}^{\smallboxtimes} {\smallboxtimes} 
(D_1 {\smallboxtimes} \mathbf{1} +  \mathbf{1} {\smallboxtimes} D_1)
\overline{D}_2 {\smallboxtimes}  \mathbf{1}^{\smallboxtimes} \\
& &+ 
\mathbf{1}^{\smallboxtimes} {\smallboxtimes} 
(\overline{D}_2 {\smallboxtimes} \mathbf{1} +  \mathbf{1} {\smallboxtimes} \overline{D}_2)
\overline{D}_2 {\smallboxtimes}  \mathbf{1}^{\smallboxtimes} . \\
\end{array}
$$
These factors are 
\begin{itemize}
\item $D_1^2 \ = \ (-\omega \circ B \circ \sigma)^2  \ = \ \omega \circ B^2 \circ \sigma = 0$,
\item $\overline{D}_2 D_1  \ = \ (-\omega^{{\smallboxtimes} 2} \circ \overline{\Delta} \circ \sigma) (-\omega \circ B \circ \sigma) \ = \  
\omega^{{\smallboxtimes} 2} \circ  \overline {\Delta}  B \circ \sigma$,
\item $\begin{array}{rcl}
(\overline{D}_2 {\smallboxtimes} \mathbf{1} +  \mathbf{1} {\smallboxtimes} \overline{D}_2)
\overline{D}_2 
& = &  ((-\omega^{{\smallboxtimes} 2} \circ \overline{\Delta} \circ \sigma) {\smallboxtimes} \mathbf{1} +  \mathbf{1} {\smallboxtimes} (-\omega^{{\smallboxtimes} 2} \circ \overline{\Delta} \circ \sigma))(-\omega^{{\smallboxtimes} 2} \circ \overline{\Delta} \circ \sigma)\\
& = &  \omega^{{\smallboxtimes} 2} \circ ( \overline{\Delta} {\smallboxtimes} \mathbf{1}^\otimes
-
\mathbf{1}^{\otimes} {\smallboxtimes} \overline{\Delta})  \overline{\Delta}  \circ \sigma
\\
& = & 0 \quad \text{ (by coassociativity of $\overline{\Delta}$)},
\end{array} 
$
\item $\begin{array}{rcl}
(D_1 {\smallboxtimes} \mathbf{1} +  \mathbf{1} {\smallboxtimes} D_1)
\overline{D}_2 
& = &  
((-\omega \circ B \circ \sigma) {\smallboxtimes} \mathbf{1} +
\mathbf{1} {\smallboxtimes} 
(-\omega \circ B \circ \sigma)) (-\omega^{{\smallboxtimes} 2} \circ \overline{\Delta} \circ \sigma)
\\
& = & -\omega^{{\smallboxtimes} 2} \circ (B  {\smallboxtimes} \mathbf{1}^\otimes +
\mathbf{1}^\otimes {\smallboxtimes} 
B  )   \overline{\Delta} \circ \sigma
\end{array} 
$
\end{itemize}
$B$ is coderivation with respect to $\Delta = \overline{\Delta} + \mathbf{1}^\otimes {\smallboxtimes} 1_\otimes + 1_\otimes {\smallboxtimes}\mathbf{1}^\otimes$, so the sum of these factors 
$$
\omega^{{\smallboxtimes} 2} \circ (\overline{\Delta} B - (B  {\smallboxtimes} \mathbf{1}^\otimes +
\mathbf{1}^\otimes {\smallboxtimes} 
B) \overline{\Delta}  ) \circ \sigma 
$$
would equal zero if $\overline{\Delta}$ were replaced by $\Delta.$  However
making the substitution $\overline{\Delta} = \Delta - \mathbf{1}^\otimes {\smallboxtimes} 1_\otimes - 1_\otimes {\smallboxtimes}\mathbf{1}^\otimes $,
 we can write 

 \vspace{.2in}

 \noindent
 $\omega^{{\smallboxtimes} 2} \circ (\overline{\Delta} B - (B  {\smallboxtimes} \mathbf{1}^\otimes +
\mathbf{1}^\otimes {\smallboxtimes} 
B) \overline{\Delta}  ) \circ \sigma  =
$
$$
\qquad
\qquad
\qquad
 \begin{array}{cl}
  & 
\omega^{{\smallboxtimes} 2} \circ [(\cancel{\Delta} -\mathbf{1}^\otimes {\smallboxtimes} 1_\otimes - 1 {\smallboxtimes}\mathbf{1}^\otimes) B 
 - (B  {\smallboxtimes} \mathbf{1}^\otimes +
\mathbf{1}^\otimes {\smallboxtimes} 
B)   (\cancel{\Delta} - \mathbf{1}^\otimes {\smallboxtimes} 1_\otimes - 1_\otimes {\smallboxtimes}\mathbf{1}^\otimes  )] \circ \sigma 
\\
 = & 
-\omega^{{\smallboxtimes} 2} \circ [(  \mathbf{1}^\otimes {\smallboxtimes} 1_\otimes + 1_\otimes {\smallboxtimes}\mathbf{1}^\otimes) B 
  -(B  {\smallboxtimes} \mathbf{1}^\otimes  +
\mathbf{1}^\otimes {\smallboxtimes} 
B)   (  \mathbf{1}^\otimes {\smallboxtimes} 1_\otimes + 1_\otimes {\smallboxtimes}\mathbf{1}^\otimes)] \circ \sigma 
\\
 = & 
-\omega^{{\smallboxtimes} 2} \circ [ 
\cancel{B {\smallboxtimes} 1_\otimes} 
+ \cancel{1_\otimes {\smallboxtimes}B}    
 - \cancel{B {\smallboxtimes} 1_\otimes} 
-  B(1)  {\smallboxtimes} \mathbf{1}^\otimes 
-\mathbf{1}^\otimes {\smallboxtimes} 
B(1) 
 - \cancel{1_\otimes {\smallboxtimes}B} 
] \circ \sigma 
\\
 = & 
\omega^{{\smallboxtimes} 2} \circ [
 b_0(1)  {\smallboxtimes} \mathbf{1}^\otimes +
\mathbf{1}^\otimes {\smallboxtimes} 
b_0(1) ] \circ \sigma 
\\
 = & 
-\omega (b_0(1))  {\smallboxtimes} \mathbf{1}^\otimes -
\mathbf{1}^\otimes {\smallboxtimes} 
(-\omega (b_0(1))   )
\\
 = & 
c {\smallboxtimes} \mathbf{1} -
\mathbf{1} {\smallboxtimes} 
c
\\
\end{array}
$$
Finally, reinserting this into the larger ${\smallboxtimes}$-product yields
$$
\mathbf{1}^{\smallboxtimes} {\smallboxtimes} (c {\smallboxtimes} \mathbf{1} - \mathbf{1} {\smallboxtimes} c) {\smallboxtimes}  \mathbf{1}^{\smallboxtimes}
= c {\smallboxtimes} \mathbf{1}^{\smallboxtimes} - \mathbf{1}^{\smallboxtimes} {\smallboxtimes} c.
$$
\end{proof}

\begin{prop}{\bf ($U_e(A)$)}
\label{proposition-defining-UeA}
The adjoint algebra is the quotient 
$$
U_e(A) =  U(A) /  I
$$ 
of the enveloping algebra by  
the two-sided ideal $I$
\begin{itemize}
\item $1_{\smallboxtimes} - \omega[\eta]$, and
\item the elements $\omega[ \sigma(a_1) \otimes \dotsm \otimes \sigma(a_{k-1}) \otimes \eta \otimes \sigma(a_{k+1}) \otimes \dotsm \otimes \sigma(a_{\ell}) ] \in X$ \\for $\ell > 1$ and $1 \leq k \leq \ell$.
\end{itemize}
The  derivation $d$ descends to a derivation on 
$U_e(A)$, and thus  $(U_e(A), d, c)$ is a curved dg-algebra.
\end{prop}
\begin{proof}
The needed property is that $dx \in I$ for any generator $x$ of $I$.  This guarantees that $dI \subseteq I:$
$$
d(uxv) =(du)xv \pm u(dx)v \pm ux(dv) \in I \quad \quad  \text{ for $x \in I$.}
$$

We can quickly dispatch with the generator $1_{\smallboxtimes} - \omega[\eta]:$ 
$$d(1_{\smallboxtimes} - \omega[\eta]) = 0-0 = 0.$$ 
So we move on to
$$
d \omega[ \sigma(a_1) \otimes  \dotsm \otimes \sigma(a_{\ell}) ]
$$
 with $\sigma(a_{k}) = \eta$ and $\ell>1$. Unrolling definitions puts this element into a form 
 about which we can reason:
$$
\begin{array}{rcl}
d \omega[ \sigma(a_1) \otimes  \dotsm \otimes \sigma(a_{\ell}) ] & =
& D_1 \omega[ \sigma(a_1) \otimes \dotsm \otimes \sigma(a_{\ell}) ] \\
&  & + \overline{D}_2 \omega[ \sigma(a_1) \otimes  \dotsm \otimes \sigma(a_{\ell}) ] 
\\
& = 
& -\omega \circ b \circ \sigma\circ \omega[ \sigma(a_1) \otimes \dotsm \otimes \sigma(a_{\ell}) ] \\
& & -\omega^{{\smallboxtimes} 2} \circ \overline{\Delta} \circ \sigma \circ \omega[ \sigma(a_1) \otimes \dotsm \dotsm \otimes \sigma(a_{\ell}) ] 
\\
&= 
& -\omega [ b (\sigma(a_1) \otimes \dotsm \otimes \sigma(a_{\ell}) ) ]\\
& & -\omega^{{\smallboxtimes} 2} [ \overline{\Delta}  (\sigma(a_1) \otimes  \dotsm \otimes \sigma(a_{\ell}) ) ]
\\
& = 
& -\omega [  b (\sigma(a_1) \otimes \dotsm \otimes \sigma(a_{\ell}) ) ]\\
& & -\omega^{{\smallboxtimes} 2} [ \overline{\Delta}  (\sigma(a_1) \otimes  \dotsm \otimes \sigma(a_{\ell}) ) ]
\\
& = 
& -\omega [ b (\sigma(a_1) \otimes \dotsm \otimes \sigma(a_{\ell}) )  ]\\ 
& & -\omega^{{\smallboxtimes} 2} [ \sum_{i=1}^{\ell-1} \sigma(a_1) \otimes \dotsm \otimes  \sigma(a_i) {\smallboxtimes} \sigma(a_{i+1}) \otimes \dotsm \otimes \sigma(a_{\ell})  ].
\\
\end{array}
$$

In this form and  when $k \not \in \{1, \ell \}$, this  reduces modulo $I$ to 
$
d \omega[ \sigma(a_1) \otimes  \dotsm \otimes \sigma(a_{\ell}) ]  \equiv_I
$
$$
 \pm \omega [ \sigma(a_1) \otimes \dotsm \otimes  \ 
  (b_2 \otimes \mathbf{1} + \mathbf{1} \otimes b_2 ) ( \sigma(a_{k-1}) \otimes  \eta \otimes \sigma(a_{k+1}))
 \  \otimes  \dotsm \otimes \sigma(a_{\ell})   ], 
$$
since the $\overline{D}_2$ terms and all the $\overline{D}_2$ terms, except those where $b_2$ ``eats'' $\eta$, lie in $I$.
These remaining $b_2$ terms equal zero:
$$
(b_2 \otimes \mathbf{1} + \mathbf{1} \otimes b_2 ) ( \sigma(a_{k-1}) \otimes  \eta \otimes \sigma(a_{k+1})) \  = \ 
(-1)^{\sigma(a_{k-1})} (-1 +1)\sigma(a_{k-1}) \otimes   \sigma(a_{k+1}) = 0.
$$

In the case when $k = 1$, 
$$
\begin{array}{rcl}
d \omega[ \eta \otimes \sigma(a_2) \otimes  \dotsm \otimes \sigma(a_{\ell}) ]  & \equiv_I  &   -\omega[ b_2(\eta \otimes \sigma(a_2))  \otimes \dotsm \otimes \sigma(a_\ell)] \\
& & -\omega^{{\smallboxtimes} 2} [\eta {\smallboxtimes}   \sigma(a_2) \otimes \dotsm \otimes \sigma(a_\ell)]
\\
& =  &   -\omega[  \sigma(a_2) \otimes \dotsm \otimes \sigma(a_\ell)] \\
& &  + \omega[\eta] {\smallboxtimes}  \omega [ \sigma(a_2) \otimes \dotsm \otimes \sigma(a_\ell)]
\\
& \equiv_I  &   -\omega[  \sigma(a_2) \otimes \dotsm \otimes \sigma(a_\ell)] \\
& &   +\omega [ \sigma(a_2) \otimes \dotsm \otimes \sigma(a_\ell)].
\\
& = &  0.
\\
\end{array}
$$
Where the last congruence uses the identity $1_{\smallboxtimes} \equiv_I  \omega [\eta]$.  The case $k = \ell$
is the same except passing $b_2$ and the second $\omega$ over the tensors picks up a global sign of
$(-1)^{\sigma(a_1) + \dotsm + \sigma(a_{\ell-1})}.$
\end{proof}

\subsection{$U_e(A)$ module identification and universality.}

\begin{thm}{\bf (module identification)}
\label{theorem-isomorphism-of-A-and-UeA-modules}
 There is an isomorphism of categories
 $$
 \operatorname{Mod_\infty^{\text{st}}}(A) \leftrightarrow  \operatorname{Mod}_\text{dg}^\text{st}(U_e(A))
 $$
 in which identifies 
$(M, b^M)$ with $M$
as a  $U_e(A)$-module with 
 \begin{itemize}
 \item differential $dm = -b_1^M(m),$ and 
 \item multiplication $m \cdot \omega[\sigma(a_1) \otimes \cdots \otimes \sigma(a_\ell)] = -(-1)^m b_\bullet^M(m \odot \sigma(a_1) \otimes \cdots \otimes \sigma(a_\ell))$.
 \end{itemize}
\end{thm}
\begin{proof}
The dg-module equations unroll exactly to $\mathbf{A}_\infty$-module equations.  These are
\begin{itemize}
\item $d^2 m  = m \cdot c$
\item $d(m \cdot x) = d(m) \cdot x + (-1)^m m \cdot d(x),$ and
\item $(m \cdot x) \cdot x' = m \cdot (x \cdot x')$.
\end{itemize}
These translate into equations for $M$ as an $A$-module:
\begin{itemize}
\item the $\mathbf{A}_\infty$-equation $b^M_\bullet(b_\bullet^M \odot \mathbf{1}^\otimes + \mathbf{1} \odot \mathbf{1}^\otimes \otimes b_\bullet \otimes \mathbf{1}^\otimes) = 0$ restricted to $M$, 
\item  the same $\mathbf{A}_\infty$-equation restricted to $M \odot A[1]^{\otimes \geq 1}$, and
\item the rule by which we extend the multiplication $M \odot X \to M$ to all of $U(A)$.
\end{itemize}
Finally, the strict-unitality conditions
\begin{itemize}
\item $-b_2^M(\mathbf{1} \odot \eta)   = \mathbf{1}$, and 
\item $b^M_\ell(\mathbf{1} \odot \mathbf{1}^{\otimes } \otimes \eta \otimes \mathbf{1}^{\otimes }) = 0$ for $\ell \neq 2$
\end{itemize}
are equivalent to this unital $U(A)$-module descending to a unital  $U_e(A)$-module.

As for morphisms, we make the identification
$$
\phi_\bullet = \phi_1 = \varphi \colon M \to M'.
$$
The conditions
\begin{itemize}
\item $\varphi(m \cdot x) = \varphi(m) \cdot x$, and
\item $\varphi(dm) = d'\varphi(m)$,
\end{itemize}
translate into 
\begin{itemize}
\item the equation  $(\delta\phi_\bullet)  =0$ restricted to $M  \odot A[1]^{\otimes \geq 1}$, and 
\item  the same equation restricted to
restricted to $M$. \end{itemize}
\end{proof}

\begin{thm}{\bf (universality of $U_e(A)$)}
 \label{theorem-universal-property-of-U_e(A)}
The assignment $A \mapsto U_e(A)$ is functorial, 
and 
$U_e$ is left adjoint to the inclusion $\mathbf{Alg}_\text{dg}^\ast \hookrightarrow \mathbf{Alg}_\infty$.
In other words, $U_e(A)$ is universal in the sense 
that there is a natural isomorphism
$$
(\omega \circ -   \circ i_\bullet) \colon 
\mathbf{Alg}_\text{dg}^\ast(U_e(-), -)
 \stackrel{\sim}{\Longrightarrow}
\mathbf{Alg}_\infty(-, -)
$$
between functors on 
 $\mathbf{Alg}_\infty^\text{op} \times  \mathbf{Alg}_\text{dg}^\ast$.
\end{thm}
\begin{proof}
Functoriality is immediate ($F \mapsto \omega F \sigma$ on variables), so consider the adjointness statement.
By abuse of notation, write $X$ for the image of  $((A[1])^{\otimes \geq 1})[-1]$
in $U_e(A);$ notice $X[1]$ is the image of $A[1]^{\otimes \geq 1}$ in $U_e(A)[1].$
Define the $\mathbf{A}_\infty$-morphism 
$$
i_\bullet \colon A[1]^{\otimes \geq 1} \to U_e(A)[1]
$$
as the map $A[1]^{\otimes \geq 1} \to X[1] \subseteq U_e(A)[1].$  

Let $A'$ be a curved dg-algebra, and consider an $\mathbf{A}_\infty$-morphism
$$
f_\bullet  \colon A[1]^{\otimes \geq 1} \to A'[1].
$$
From this we can define 
$$
\mathfrc{f} \colon  U_e(A) \to A'
$$
by $\mathfrc{f}(x) = (\omega \circ f_\bullet  \circ \sigma)(x)$ for $x \in X$.  These maps fit in a commutative diagram
$$
\begin{tikzcd}
(A[1])^{\otimes \geq 1} \arrow{dr}[swap]{f_\bullet}  \arrow{r}{i_\bullet}& U_e(A)[1] 
 \arrow{d}{\sigma \circ \mathfrc{f} \circ \omega} \\
&A'[1]. \\
\end{tikzcd}
$$
Furthermore, given $\mathfrc{f}$ one can use this diagram to define $f_\bullet.$

The remaining question is whether or not the maps commute with operators $B$.
This depends only on checking that $i_\bullet$ is an $\mathbf{A}_\infty$-morphism.  Indeed,
taking $\mathfrc{f} = $ identity map on  $U_e(A)$ shows the necessity,  and sufficiency 
follows from the fact that $(i_\bullet)^\diamond$ is injective and
$$
\begin{array}{rcl}
(B (\sigma \circ \mathfrc{f} \circ \omega)^{\diamond} - (\sigma \circ \mathfrc{f} \circ \omega)^{\diamond} B)  (i_\bullet)^\diamond & 
= & B  (\sigma \circ \mathfrc{f} \circ \omega)^{\diamond}  (i_\bullet)^\diamond - (\sigma \circ \mathfrc{f} \circ \omega)^{\diamond}  (i_\bullet)^\diamond B \\
& = &  B (f_\bullet)^\diamond -  (f_\bullet)^\diamond B.
\end{array}
$$

Finally, layers of notation constitute the only difficulty in 
 verifying $i_\bullet$ is an $\mathbf{A}_\infty$-morphism. 
To be clear, the output of $(i_\bullet)^{\diamond}$ on an element $\sigma(a_1) \otimes \dotsm \otimes \sigma(a_k)$
is the sum of the $2^{k-1}$ terms $\sigma(a_1)  \oslash \dotsm \oslash \sigma(a_k)$ where $\oslash \in \{ \otimes, \diamond \}$.
We must consider $B (i_\bullet)^{\diamond} - (i_\bullet)^{\diamond} B = $
$$
[
\mathbf{1}^\diamond \diamond (- \sigma \circ d \circ \omega) \diamond \mathbf{1}^\diamond + 
\mathbf{1}^\diamond \diamond (- \sigma \circ ( \mathbf{1} \cdot \mathbf{1} ) \circ \omega^{\diamond 2}) \diamond \mathbf{1}^\diamond
] (i_\bullet)^{\diamond}
-
(i_\bullet)^{\diamond}
[
\mathbf{1}^\otimes \otimes b_\bullet \otimes \mathbf{1}^\otimes
] .
$$
In isolation, 
$\mathbf{1}^\diamond \diamond (- \sigma \circ d \circ \omega) \diamond \mathbf{1}^\diamond + 
\mathbf{1}^\diamond \diamond (- \sigma \circ ( \mathbf{1} \cdot \mathbf{1} ) \circ \omega^{\diamond 2}) \diamond \mathbf{1}^\diamond$
restricted to the output of $(i_\bullet)^\diamond$ equals
$$
\begin{array}{cl}
& \mathbf{1}^\diamond \diamond (- \sigma \circ (-\omega \circ B \circ \sigma - \omega^{{\smallboxtimes} 2} \circ \overline{\Delta} \circ \sigma)
\circ \omega) \diamond \mathbf{1}^\diamond  \\
& + \mathbf{1}^\diamond \diamond (- \sigma \circ ( \mathbf{1} \cdot \mathbf{1} ) \circ \omega^{\diamond 2}) \diamond \mathbf{1}^\diamond \\
= &  \mathbf{1}^\diamond \diamond (- \sigma \circ (-\omega \circ B \circ \sigma) 
\circ \omega) \diamond \mathbf{1}^\diamond  \\
& + \mathbf{1}^\diamond \diamond (- \sigma \circ (
- \omega^{{\smallboxtimes} 2} \circ \overline{\Delta} \circ \sigma)
\circ \omega) \diamond \mathbf{1}^\diamond  \\
& + \mathbf{1}^\diamond \diamond (- \sigma \circ ( \mathbf{1} \cdot \mathbf{1} ) \circ \omega^{\diamond 2}) \diamond \mathbf{1}^\diamond \\
= & \mathbf{1}^\diamond \diamond B \diamond \mathbf{1}^\diamond \\
 & +\mathbf{1}^\diamond \diamond ( \sigma \circ (
 \omega^{{\smallboxtimes} 2} \circ \overline{\Delta})) \diamond \mathbf{1}^\diamond  \\
& - \mathbf{1}^\diamond \diamond (\sigma \circ ( \mathbf{1} \cdot \mathbf{1} ) \circ \omega^{\diamond 2}) \diamond \mathbf{1}^\diamond .\\
\end{array}
$$
The last two lines in the expression cancel.  This can been seen by the fact that their output is in tensors of the form 
$\sigma(a_1)  \oslash \dotsm \oslash \sigma(a_k)$ where $\oslash \in \{ \otimes, \diamond, {\smallboxtimes} \}$ with exactly one $\oslash = {\smallboxtimes}.$
Each such tensor is obtained two ways with opposite signs: one by changing ${\smallboxtimes}$ to $\diamond,$ and the other by changing ${\smallboxtimes}$ to $\otimes.$

It remains to compare  $[\mathbf{1}^\diamond \diamond B \diamond \mathbf{1}^\diamond]  (i_\bullet)^\diamond$ and $(i_\bullet)^{\diamond}
[
\mathbf{1}^\otimes \otimes b_\bullet \otimes \mathbf{1}^\otimes
]$.  Again these cancel, as can be seen by the fact that an output tensor is obtained by putting in first either the $\diamond$'s or the $\mathbf{1}^\otimes \otimes b_\bullet \otimes \mathbf{1}^\otimes.$
\end{proof}


\section{Module adjunctions and the quasi-equivalence $Q \sim \mathbf{1}$.} 
\label{section:module-adjunctions}

This section defines a bimodule structure on $U_e(A)$ (Lemma \ref{lemma-UA-bimod}), uses this to define the functor $Q$ and prove that it is adjoint to inclusion of the strict subcategory (Theorem \ref{theorem-definition-of-Q}),
and finally shows that $Q$ and $\mathbf{1}$ are quasi-equivalent functors on $\operatorname{Mod_\infty}(A)$ (Theorem \ref{theorem-Q-1-quasi-equivalence}).

\begin{lem} {\bf (the bimodule $U_e(A)$)}
\label{lemma-UA-bimod}
The adjoint algebra is a strictly unital $A-A$-bimodule defined by  
$$
b_{\bullet, \bullet} \colon (A[1])^\otimes \odot U_e(A) \odot (A[1])^\otimes \to U_e(A)
$$
where
$$
b_{\bullet, \bullet}  = d + \omega^+ \smallboxtimes \mathbf{1}^{\smallboxtimes}  -   \mathbf{1}^{\smallboxtimes}   \smallboxtimes \omega^+ 
$$
for the shift-multiplication maps
\begin{itemize}
\item $\omega^+ \smallboxtimes \mathbf{1}^{\smallboxtimes} \colon (A[1])^{\otimes \geq 1} \odot U_e(A) \to U_e(A)$,
and
\item $\mathbf{1}^{\smallboxtimes} \smallboxtimes \omega^+ \colon U_e(A) \odot (A[1])^{\otimes \geq 1} \to U_e(A).$
\end{itemize}
\end{lem}

\begin{proof}
We take a slightly indirect route in this proof.
Verification that these maps give $U_e(A)$ the structure of a bimodule amounts to 
checking $(B^{U_e(A)})^2 = 0$ for 
$$
B^{U_e(A)} = 
B^A \odot \mathbf{1}^{\smallboxtimes} \odot \mathbf{1}^\otimes
+
\mathbf{1}^\otimes \odot b_{\bullet, \bullet} \odot \mathbf{1}^\otimes
+
\mathbf{1}^\otimes \odot \mathbf{1}^{\smallboxtimes} \odot B^A.
$$
Rather than computing  $(B^{U_e(A)})^2$ directly, we identify 
$(A[1])^\otimes \odot U(A) \odot (A[1])^\otimes$  with a differential 
ideal in a dga $\Omega$ whose differential satisfies $(d^\Omega)^2 = 0$.
The remaining work is to show $B^{U(A)}$ (defined analogously to $B^{U_e(A)}$) agrees with $d^\Omega$
under this identification, and finally to check that $B^{U(A)}$
descends to $B^{U_e(A)}$.

Our dg-algebra  is $\Omega = (Y^{\smallboxtimes}, d^\Omega),$
where $Y = ((A[1])^{\otimes})[-1]$, and
the differential 
$d^\Omega$ defined in a way similar to $U(A)$ by prolonging $D^\Omega = D_1 + D_2$ with
\begin{itemize}
\item $D_1 = - \omega \circ B^A \circ \sigma $ as before, and 
\item $D_2 = -\omega^{\smallboxtimes 2} \circ \Delta \circ \sigma$.
\end{itemize}
Now, $(d^{\Omega})^2 = 0$ reduces to the fact $B^A$ is a coderivation for $\Delta.$

The first bimodule under consideration is $\Omega$ itself; later we will find that $U(A)$
is a sub-bimodule.  For $\Omega$, we consider  
  the map
$$
\omega \smallboxtimes \mathbf{1}^{\smallboxtimes} \smallboxtimes \omega \colon A[1]^\otimes \odot \Omega \odot A[1]^\otimes \to  Y \smallboxtimes \Omega \smallboxtimes Y.
$$
This is an $S$-linear isomorphism with inverse $- (\sigma  \odot \mathbf{1}^{\smallboxtimes} \odot \sigma)$.  To be sure about this, notice that the element $y \smallboxtimes y' = y \smallboxtimes 1_{\smallboxtimes} \smallboxtimes y \in Y \smallboxtimes \Omega \smallboxtimes Y$
maps to $-(-1)^{y}\sigma(y) \odot 1_{\smallboxtimes} \odot \sigma(y').$ 
The definition
$$
B^\Omega = - (\sigma  \odot \mathbf{1}^{\smallboxtimes} \odot \sigma) \circ d^\Omega \circ (\omega \smallboxtimes \mathbf{1}^{\smallboxtimes} \smallboxtimes \omega).
$$
guarantees $(B^\Omega)^2 = 0$.

Some work remains to show that $B^\Omega$ is of the appropriate form to define a bimodule.
Carefully expanding the definition of $B^\Omega$ leads to 
$$
B^\Omega = B^A \odot \mathbf{1}^{\smallboxtimes} \odot \mathbf{1}^\otimes + \mathbf{1}^\otimes \odot b_{\bullet, \bullet} \odot \mathbf{1}^\otimes + \mathbf{1}^\otimes \odot \mathbf{1}^{\smallboxtimes} \odot B^A.
$$
where $b_{\bullet, \bullet} = d + \omega^+ \smallboxtimes \mathbf{1}^{\smallboxtimes} - \mathbf{1}^{\smallboxtimes}  \smallboxtimes \omega^+$  Not coincidently, this is  same formula  as proposed for $U_e(A).$

To begin we first split $d^{\Omega}$ into its linear and quadratic parts:
$$
d^{\Omega} = \mathbf{1}^{\smallboxtimes} \smallboxtimes D_1 \smallboxtimes  \mathbf{1}^{\smallboxtimes} 
+  
\mathbf{1}^{\smallboxtimes} \smallboxtimes D_2 \smallboxtimes \mathbf{1}^{\smallboxtimes}.
$$
Restricted to $Y \smallboxtimes Y^{\smallboxtimes} \smallboxtimes Y = Y \smallboxtimes Y \oplus Y \smallboxtimes Y \smallboxtimes Y \oplus \dotsm$,
the first summand yields
$$
\begin{array}{rccl}
\mathbf{1}^{\smallboxtimes} \smallboxtimes D_1 \smallboxtimes  \mathbf{1}^{\smallboxtimes} 
& = & & D_1 \smallboxtimes \mathbf{1}^{\smallboxtimes} \\
& & + &  \mathbf{1}^\otimes \smallboxtimes \mathbf{1}^{\smallboxtimes} \smallboxtimes D_1 \smallboxtimes  \mathbf{1}^{\smallboxtimes} \smallboxtimes \mathbf{1}^\otimes \\
& & + & \mathbf{1}^{\smallboxtimes} \smallboxtimes D_1 .
\end{array} 
$$
So for $-(\sigma  \odot \mathbf{1}^{\smallboxtimes} \odot \sigma) \circ (\mathbf{1}^{\smallboxtimes} \smallboxtimes D_1 \smallboxtimes  \mathbf{1}^{\smallboxtimes} ) \circ (\omega \smallboxtimes \mathbf{1}^{\smallboxtimes} \smallboxtimes \omega)$
we have
$$
 (- \sigma \circ D_1 \circ \omega) \odot \mathbf{1}^{\smallboxtimes} \odot \mathbf{1}^\otimes
 +   \mathbf{1}^\otimes \odot \mathbf{1}^{\smallboxtimes} \smallboxtimes D_1 \smallboxtimes  \mathbf{1}^{\smallboxtimes} \odot \mathbf{1}^\otimes
+ \mathbf{1}^\otimes \odot \mathbf{1}^{\smallboxtimes} \odot (-\sigma \circ D_1 \circ \omega) 
$$
which equals
$$
B^A \odot \mathbf{1}^{\smallboxtimes} \odot \mathbf{1}^\otimes
 +   \mathbf{1}^\otimes \odot \mathbf{1}^{\smallboxtimes} \smallboxtimes D_1 \smallboxtimes  \mathbf{1}^{\smallboxtimes} \odot \mathbf{1}^\otimes 
+ \mathbf{1}^\otimes \odot \mathbf{1}^{\smallboxtimes} \odot B^A.
$$
So the question remains with the $D_1$ term.

The second summand expands slightly differently than the first.
Restricted to $Y \smallboxtimes Y^{\smallboxtimes} \smallboxtimes Y$  we can again organize the sum as
$$
\begin{array}{rccl}
\mathbf{1}^{\smallboxtimes} \smallboxtimes D_2 \smallboxtimes  \mathbf{1}^{\smallboxtimes} 
& = & & D_2 \smallboxtimes \mathbf{1}^{\smallboxtimes} \\
& & + &  \mathbf{1} \smallboxtimes \mathbf{1}^{\smallboxtimes} \smallboxtimes D_2 \smallboxtimes  \mathbf{1}^{\smallboxtimes} \smallboxtimes \mathbf{1} \\
& & + & \mathbf{1}^{\smallboxtimes} \smallboxtimes D_2 .
\end{array} 
$$
For $-(\sigma  \odot \mathbf{1}^{\smallboxtimes} \odot \sigma) \circ (\mathbf{1}^{\smallboxtimes} \smallboxtimes D_2 \smallboxtimes  \mathbf{1}^{\smallboxtimes} ) \circ (\omega \smallboxtimes \mathbf{1}^{\smallboxtimes} \smallboxtimes \omega)$
the expansion is more subtle.  The middle term is simply
$$
\mathbf{1}^\otimes \odot \mathbf{1}^{\smallboxtimes} \smallboxtimes D_2 \smallboxtimes  \mathbf{1}^{\smallboxtimes} \odot \mathbf{1}^\otimes,
$$
but we will have to come back to it after expanding the ends.
The ends become
$$
- ((\sigma \odot \mathbf{1}^{\smallboxtimes}) \circ D_2 \circ \omega) \smallboxtimes \mathbf{1}^{\smallboxtimes} \odot \mathbf{1}^{\otimes}
- \mathbf{1}^{\otimes} \odot \mathbf{1}^{\smallboxtimes} \smallboxtimes ((\mathbf{1}^{\smallboxtimes} \odot \sigma) \circ D_2 \circ \omega)
$$
and then 
$$
 (\mathbf{1}^\otimes \odot \omega) \smallboxtimes \mathbf{1}^{\smallboxtimes} \odot \mathbf{1}^{\otimes}
- \mathbf{1}^{\otimes} \odot \mathbf{1}^{\smallboxtimes} \smallboxtimes (\omega \odot \mathbf{1}^{\otimes})
$$
where $\omega \colon A[1]^\otimes \to Y$.

These computations prove the bimodule structure on $\Omega$ with
$$
b_{\bullet , \bullet} = d^\Omega + \omega {\smallboxtimes} \mathbf{1}^{\smallboxtimes} - \mathbf{1}^{\smallboxtimes} {\smallboxtimes} \omega.
$$
The final formulas come from the identity
$$
d^\Omega = d - \omega(1_\otimes) \smallboxtimes \mathbf{1}^{\smallboxtimes} +  \mathbf{1}^{\smallboxtimes} \smallboxtimes \omega(1_\otimes)
$$
whose substitution into this expression leaves us with what we want:
$$
b_{\bullet , \bullet} = d + \omega^+ {\smallboxtimes} \mathbf{1}^{\smallboxtimes} - \mathbf{1}^{\smallboxtimes} {\smallboxtimes} \omega^+.
$$

In this form, we know that  $U(A) \subseteq \Omega$ is carried to itself under $B^\Omega$.  The algebra $U_e(A)$ is a quotient of $U(A)$ by 
an ideal.  Both $d$ and multiplication by elements of $\Omega$ make sense on this quotient so $B^{\Omega}$ descends with the same formula.
\end{proof}

\begin{thm}{\bf (module adjunctions)}
 \label{theorem-definition-of-Q}
  \label{theorem-universal-Q}
The functor 
$$
Q_A = - \overset{\infty}{\otimes} U_e(A) \colon 
   \operatorname{Mod}_\infty(A)  
\to
  \operatorname{Mod}_\infty(A).
$$
induces a $dg$-adjunction
$$
(- \circ \lambda_\bullet) \colon \operatorname{Mod}_\text{dg}(U_e(A))(Q_A(-), -)  \stackrel{\sim}{\Longrightarrow}  \operatorname{Mod}_\infty(A)(-, -)
$$
and consequently an adjunction
$$
 \operatorname{Mod}_\infty^\text{st}(A)(Q_A(-), -) \stackrel{\sim}{\Longrightarrow}
Z^0(\operatorname{Mod}_\infty(A)(-, -)).
$$
\end{thm}
\begin{proof}
This is almost an immediate consequence of Lemma \ref{lemma-UA-bimod}.  
It becomes obvious with the notation
$$
\lambda_\bullet \colon M \odot A[1]^\otimes \to  M \odot A[1]^\otimes \odot U_e(A) 
$$
identifying $M \odot A[1]^\otimes \to  M \odot A[1]^\otimes \odot 1_{\smallboxtimes}$.
The $S$-modules $
\operatorname{Mod}_\text{dg}(U_e(A))(Q_A(M), N)  
$ 
and 
$\operatorname{Mod}_\infty(A)(M, N)$
are identified this way.  

Agreement as complexes follows from the closure of  $\lambda_\bullet,$
and this requires a computation to check the vanishing of 
the map 
$$
M \odot A[1]^{\otimes} \to M \odot A[1]^\otimes \odot U_e(A) \odot A[1]^\otimes
$$
given by 
$$
B^{Q_A(M)} \circ (\lambda_\bullet \odot \mathbf{1}^\otimes) - (\lambda_\bullet \odot \mathbf{1}^\otimes)  \circ B^M.
$$
Notice that $\lambda_\bullet \odot \mathbf{1}^{\otimes}$ has the effect of inserting a ``$\odot 1_{\smallboxtimes} \odot$'' into all possible spots.
In the expression $(\lambda_\bullet \odot \mathbf{1}^\otimes)  \circ B^M$ this can happen to the right or left of the $b$ that appears.
Similarly, $B^{Q_A(M)} \circ (\lambda_\bullet \odot \mathbf{1}^\otimes)$ has terms with ``$\odot 1_{\smallboxtimes} \odot$'' on either the right or left of the $b,$
and also terms where the ``$\odot 1_{\smallboxtimes} \odot$''  was ``eaten'' by $B^{Q_A(M)}$.  

The $b$-separating terms cancel, and we are left with the terms 
which eat ``$\odot 1_{\smallboxtimes} \odot$''. 
Since $d(1_{\smallboxtimes}) = 0,$ the relevant operator is the
$$
 \omega^+ \smallboxtimes \mathbf{1}^{\smallboxtimes 0} - \mathbf{1}^{\smallboxtimes 0}  \smallboxtimes \omega^+.
$$
part of $$
b_{\bullet , \bullet} = d + \omega^+ {\smallboxtimes} \mathbf{1}^{\smallboxtimes} - \mathbf{1}^{\smallboxtimes} {\smallboxtimes} \omega^+.
$$
Finally, a term of the form 
$$
m \odot \sigma(a_1) \otimes \dotsm \otimes \sigma(a_k) \odot \omega[\sigma(a_{k+1}) \otimes \dotsm \otimes \sigma(a_l)] \odot \sigma(a_{l+1}) \otimes \dotsm \otimes \sigma(a_n)
$$
can arise in two ways with opposite signs depending on whether the original element was
$$
m \odot \sigma(a_1) \otimes \dotsm \otimes \sigma(a_k) \odot 1_{\smallboxtimes} \odot \sigma(a_{k+1}) \otimes \dotsm \otimes \sigma(a_l) \otimes\sigma(a_{l+1}) \otimes \dotsm \otimes \sigma(a_n)
$$
or
$$
-m \odot \sigma(a_1) \otimes \dotsm \otimes \sigma(a_k) \odot \sigma(a_{k+1}) \otimes \dotsm \otimes \sigma(a_l) \odot 1_{\smallboxtimes} \odot \sigma(a_{l+1}) \otimes \dotsm \otimes \sigma(a_n).
$$ 
\end{proof}

\begin{thm}{\bf (the quasi-equivalence $Q_A \sim \mathbf{1}$)}
\label{theorem-Q-1-quasi-equivalence}
On the category $H^0(\operatorname{Mod_\infty}(A))$, 
the functor $Q_A$ is naturally isomorphic to the identity functor  $\mathbf{1}$.
This natural isomorphism is realized by $\lambda_\bullet$ and an inverse natural transformation 
$$\epsilon_\bullet \colon Q_A \Rightarrow \mathbf{1}. $$ 
For a given module,  $\epsilon_\bullet$ is the image of 
 the identity map $M \to M$
under the adjunction isomorphism of Theorem \ref{theorem-universal-Q}.
\end{thm}
\begin{proof}
The statement requires that we show that the compositions $(\epsilon_\bullet \odot \mathbf{1}^\otimes)\circ (\lambda_\bullet \odot \mathbf{1}^\otimes)$ and $(\lambda_\bullet \odot \mathbf{1}^\otimes)\circ (\epsilon_\bullet \odot \mathbf{1}^\otimes)$ differ from the identity by a boundary.  In the first case, the definition of $\epsilon_\bullet$ via the adjunction guarantees the composition equals the identity exactly.  The second case is not so simple.  For this we introduce a homotopy $H$ and consider the equality
$$
BH + HB = \mathbf{1} - \Lambda E
$$ 
where $B = B^{Q_A}$, $\Lambda = (\lambda_\bullet \odot \mathbf{1}^\otimes)$, and $E =  (\epsilon_\bullet \odot \mathbf{1}^\otimes).$  Verification of this identity splits into two cases depending on the argument 
of these operators, and the analysis of the second of these cases requires systematic classification of the kind of terms output from $BH+HB.$

To understand 
the difference
$$
\mathbf{1} - \Lambda E 
\quad 
\colon 
\quad 
Q_A(M) \to Q_A(M)
$$
we must describe concretely the composition $\Lambda E.$
The map $\epsilon_\bullet \odot \mathbf{1}^\otimes $ first projects 
$$
M \odot A[1]^\otimes \odot U_e(A) \odot A[1]^\otimes \to M \odot 1_\otimes \odot U_e(A) \odot A[1]^\otimes 
$$
and then multiplies $U_e(A)$ against $M$:
$$
M \odot 1_\otimes \odot U_e(A) \odot A[1]^\otimes 
\to
M  \odot A[1]^\otimes.
$$
The right-hand $A[1]^\otimes$ just comes along for the ride.
Following $\epsilon_\bullet \odot \mathbf{1}^{\otimes}$ by $\lambda_\bullet \odot \mathbf{1}^{\otimes}$
has the total effect on an element\footnote{For the remainder of the proof, we write Greek letters for monomials: e.g. $\chi = x_{1} \smallboxtimes \dotsm \smallboxtimes x_k$ and
$\alpha = a_{1} \otimes \dotsm \otimes a_\ell$.  In addition, we add $'$ and $''$ to indicate splittings.} $m \odot  1_\otimes \odot \chi \odot \alpha$  in $M \odot 1_\otimes \odot U_e(A) \odot A[1]^\otimes$
of
$$
\Lambda E(m \odot  1_\otimes \odot \chi \odot \alpha) = 
\sum_{\alpha' \otimes \alpha'' = \alpha} m \cdot \chi \odot  \alpha' \odot 1_{\smallboxtimes} \odot \alpha''.
$$

With this description, the problem becomes to find a homotopy $H$ such that
the output of $BH+HB$ on an element $m \odot \alpha_1 \odot \chi \odot \alpha_2$ is
\begin{itemize}
\item $m \odot  \alpha_1 \odot \chi \odot \alpha_2 $ \quad when $\alpha_1  \neq 1_\otimes$, and
\item $m \odot  1_\otimes \odot \chi \odot \alpha  \quad - \quad \sum\limits_{\substack{\alpha' \neq 1_\otimes \\ \alpha' \otimes \alpha'' = \alpha}} \ m \cdot u \odot  \alpha' \odot 1_{\smallboxtimes}\odot \alpha''$
when $\alpha_1 = 1_\otimes$ and $\alpha_2 = \alpha.$
\end{itemize}
To this end, we set 
$$
H = h_\bullet \odot \mathbf{1}^\otimes \colon M \odot A[1]^\otimes \odot U_e(A) \odot A[1]^\otimes \to M \odot A[1]^\otimes \odot U_e(A) \odot A[1]^\otimes
$$
given by first projecting onto 
$M \odot 1_\otimes  \odot U_e(A)   \odot A[1]^\otimes$,  then marching  the variables of a monomial in $U_e(A)$ through and across the left $A[1]^\otimes$-slot.  For example,
$$
H(m \odot 1_\otimes \odot x_1 {\smallboxtimes}\dotsm {\smallboxtimes}x_k \odot \alpha)
=
\sum_{i = 1}^\ell (-1)^{|m| + |x_1| + \dotsm + |x_{i-1}|}  \ m \cdot (x_1 {\smallboxtimes}\dotsm \smallboxtimes x_{i-1}) \odot \sigma(x_i) \odot x_{i+1} {\smallboxtimes}\dotsm \smallboxtimes x_{\ell}  \odot \alpha.
$$
This leads to the expression
$$
H= \mathbf{1} \cdot  \mathbf{1}^{\smallboxtimes} \odot \sigma \odot \mathbf{1}^{\smallboxtimes} \odot \mathbf{1}^\otimes.
$$

With a view toward  the two cases above,  we feed elements through the 
boundary of $H$:
$$
BH+ HB.
$$
The case of an element $m \odot  \alpha' \odot \chi \odot \alpha'' $  when $\alpha' \neq 1_\otimes$ can be disposed of quickly.  The left-hand summand of the boundary 
annihilates this element, and the only terms output from  $B$ not annihilated by $H$
are
$$
\begin{array}{c}
b_\bullet^M(m \otimes \alpha') \odot 1_\otimes \odot  \chi \odot \alpha'' + 
(-1)^m m \odot  1_\otimes  \odot \omega[\alpha'] \smallboxtimes \chi \odot \alpha'' = \\
-(-1)^m m \cdot  \omega[\alpha'] 
\odot 1_\otimes \odot  \chi \odot \alpha'' + 
(-1)^m m \odot  1_\otimes  \odot \omega[\alpha'] \smallboxtimes \chi \odot \alpha''  = \\
(-1)^m (m \odot  1_\otimes  \odot \omega[\alpha'] \smallboxtimes \chi \odot \alpha'' 
- m \cdot  \omega[\alpha'] 
\odot 1_\otimes \odot  \chi \odot \alpha'' )
\end{array}
$$
Plugging this into $h\odot \mathbf{1}^\otimes$ yields
$$
\begin{array}{c}
m \odot  \alpha' \odot \chi \odot \alpha'' + (-1)^m  (h \odot \mathbf{1}^\otimes) (m \cdot \omega[\alpha'] \odot  1_\otimes  \odot \omega[\alpha']  \chi \odot \alpha'' 
- m \cdot  \omega[\alpha'] 
\odot 1_\otimes \odot  \chi \odot \alpha'' )  \\
 = m \odot  \alpha' \odot \chi \odot \alpha'' 
\end{array}
$$
by the identity 
$$
(h \odot \mathbf{1}^\otimes)(m \odot 1_\otimes \odot x \smallboxtimes \chi \odot \alpha) = (-1)^m (m \odot \sigma(x) \odot \chi \odot \alpha) + (h \odot \mathbf{1}^\otimes)(m \cdot x \odot 1_\otimes \odot  \chi \odot \alpha).
$$

{\bf Classification by 6 types.}  
Analysis of the output of arguments which have a  $1_\otimes$ in the left-hand $A[1]^\otimes$-slot 
requires splitting the output into cases.
The terms which appear
fall into six distinct basic types .  These types can be identified by what lies between 
the two leftmost $\odot$ symbols.  The labels of the classification are:
$$
\odot 1_\otimes \odot, \quad \odot \sigma[x_i] \odot, \quad \odot B^A(\sigma[x_i]) \odot, \quad \odot \sigma[x_i'] \odot, \quad \odot \sigma[x_i''] \odot, \quad \text{ and }\quad \odot \alpha' \odot.
$$
With the convention that $x_i' \smallboxtimes x_i''$ is a term in $\overline{D}_2(x_i),$ and $\alpha = \alpha' \otimes \alpha''$ where $\alpha' \neq 1_\otimes,$ a moment's reflection verifies these types partition the  output terms. 

{\bf The type $\odot 1_\otimes \odot$.} A term of the type  $\odot 1_\otimes \odot$ can only be output by $BH$, and all cancel except the term 
$$
 m \odot 1_\otimes \odot x_{1} \smallboxtimes \dotsm \smallboxtimes x_{k} \odot \alpha. 
$$ 
This term is output only by $\omega^+$ evaluated on $m \odot \sigma[x_{1}]  \odot x_2 \smallboxtimes \dotsm \smallboxtimes x_{k} \odot \alpha$.  Cancellation of the other terms happens because they can
arise either
under the action of $b^M_\bullet$
$$
m \cdot (x_1 \smallboxtimes \dotsm \smallboxtimes x_{i-1}) \odot \sigma[x_i] \odot x_{i+1} \smallboxtimes \dotsm \smallboxtimes x_{k} \odot \alpha \mapsto 
\pm m \cdot (x_1 \smallboxtimes \dotsm \smallboxtimes x_{i-1} \smallboxtimes x_i) \odot 1_\odot \odot x_{i+1} \smallboxtimes \dotsm \smallboxtimes x_{k} \odot \alpha 
$$
or under the action of $\omega^+$
$$
m \cdot (x_1 \smallboxtimes \dotsm \smallboxtimes x_{i}) \odot \sigma[x_{i+1}] \odot x_{i+2} \smallboxtimes \dotsm \smallboxtimes x_{k} \odot \alpha \mapsto 
\pm m \cdot (x_1 \smallboxtimes \dotsm \smallboxtimes x_{i}) \odot 1_\odot \odot x_{i+1} \smallboxtimes x_{i+2} \smallboxtimes \dotsm \smallboxtimes x_{k} \odot \alpha .
$$
These two terms appear with opposite signs because  $x_i$ is multiplied against the module term with $b^M_\bullet \circ \sigma$ in the first case and $\cdot  =  -b^M_\bullet \circ \sigma$
in the second.

{\bf The type $\odot \sigma[x_i] \odot$.} All of these terms cancel, but for two different reasons, so we must split this type into two sub-types: 
\begin{itemize}
\item $m \cdot (x_1 \smallboxtimes \dotsm \smallboxtimes x_{i-1}) \odot \sigma[x_i] \odot \dotsm$  \quad whose left segment is unaffected by $B$, and
\item $\dotsm \odot \sigma[x_i]  \odot x_{i+1} \smallboxtimes \dotsm \smallboxtimes x_{k} \odot \alpha$ \quad  \ \ whose right segment is unaffected by $B$. 
\end{itemize}
The first subtype cancels because one is simply putting $B$ on the right segment either before or after $H$ moves $\sigma[x_i]$ between the $\odot$'s, so 
both $BH$ and $HB$ produce these terms and with opposite signs.
The second subtype has one contribution from $HB$:
$$
\pm d(m \cdot (x_1 \smallboxtimes \dotsm \smallboxtimes x_{i-1})) \odot \sigma[x_i] \odot x_{i+1} \smallboxtimes \dotsm \smallboxtimes x_{k} \odot \alpha
$$
the contribution from $HB$ is
$$
\pm d(m) \cdot (x_1 \smallboxtimes \dotsm \smallboxtimes x_{i-1}) \odot \sigma[x_i] \odot x_{i+1} \smallboxtimes \dotsm \smallboxtimes x_{k} \odot \alpha
$$
and terms of the form
$$
\pm m \cdot (x_1 \smallboxtimes \dotsm \smallboxtimes  d(x_{i'}) \smallboxtimes \dotsm \smallboxtimes  x_{i-1}) \odot \sigma[x_i] \odot x_{i+1} \smallboxtimes \dotsm \smallboxtimes x_{k} \odot \alpha.
$$
Considered together, the fact that $d$ is a derivation guarantees all these terms cancel, and one can check the signs most easily by drawing the string diagram for the operators
\begin{itemize}
\item $(d \odot \mathbf{1}^{\otimes} \odot \mathbf{1}^{\smallboxtimes}  \odot \mathbf{1}^{\otimes}) ( \mathbf{1} \cdot  \mathbf{1}^{\smallboxtimes} \odot \sigma \odot \mathbf{1}^{\smallboxtimes} \odot \mathbf{1}^\otimes)$,
\item $(\mathbf{1} \cdot  \mathbf{1}^{\smallboxtimes} \odot \sigma \odot \mathbf{1}^{\smallboxtimes} \odot \mathbf{1}^\otimes) (d \odot \mathbf{1}^{\otimes} \odot \mathbf{1}^{\smallboxtimes}  \odot \mathbf{1}^{\otimes})$, and 
\item $( \mathbf{1} \cdot  \mathbf{1}^{\smallboxtimes} \odot \sigma \odot \mathbf{1}^{\smallboxtimes} \odot \mathbf{1}^\otimes) (\mathbf{1} \odot \mathbf{1}^{\otimes} \odot d \odot \mathbf{1}^{\otimes}) $
\end{itemize}
where the only sign happens when $d$ passes $\sigma$ in the second two cases.

{\bf The type $\odot B^A(\sigma[x_i]) \odot$.} These terms cancel away in  essentially the same way as those of  type $\odot \sigma[x_i] \odot$ for which $B$ is on the right segment.  In brief, $B^A$
is put onto $x_i$ either before or after $H$ moves it between the $\odot$'s.

{\bf The type $\odot \sigma[x_i'] \odot$.} A term of this type is produced by $BH$ and $HB$, and cancels.  In the $BH$ case, $H$ moves $x_i$ between the  $\odot$'s and then $\omega^+$ splits it.
For $HB$,  $\overline{D}_2$  splits $x_i$ and then $H$ moves $x_i'$ between the  $\odot$'s. For signs, the relevant string diagrams begin from 
\begin{itemize}
\item $(\mathbf{1} \odot \mathbf{1}^\otimes  \odot \omega^+ \smallboxtimes \mathbf{1}^{\smallboxtimes} \odot \mathbf{1}^\otimes)(\mathbf{1} \cdot  \mathbf{1}^{\smallboxtimes} \odot \sigma \odot \mathbf{1}^{\smallboxtimes} \odot \mathbf{1}^\otimes)$ and
\item $(\mathbf{1} \cdot  \mathbf{1}^{\smallboxtimes} \odot \sigma \odot \mathbf{1}^{\smallboxtimes} \odot \mathbf{1}^\otimes)(\mathbf{1} \odot \mathbf{1}^\otimes \odot \mathbf{1}^{\smallboxtimes} \smallboxtimes (-\omega^{\smallboxtimes 2} \circ \overline{\Delta} \circ \sigma) \smallboxtimes \mathbf{1}^{\smallboxtimes} \odot \mathbf{1}^\otimes)$
\end{itemize}
and end at 
$$
\mathbf{1} \cdot \mathbf{1}^{\smallboxtimes} \odot (\mathbf{1}^\otimes \odot \omega^+) \sigma \smallboxtimes \mathbf{1}^{\smallboxtimes} \odot \mathbf{1}^{\otimes}
= \mathbf{1} \cdot  \mathbf{1}^{\smallboxtimes} \odot \sigma (\omega\odot \omega \circ \overline{\Delta} \circ \sigma) \smallboxtimes \mathbf{1}^{\smallboxtimes} \odot \mathbf{1}^{\otimes}.$$
Neither have any crossings, so the cancellation comes from the sign on $-\omega^{\smallboxtimes 2}$ in the $HB$ term.

{\bf The type $\odot \sigma[x_i''] \odot$.} This type is quite similar to $\odot \sigma[x_i'] \odot$, and we have cancellation again.  As before such a term 
 is produced by $BH$ and $HB$. 
 The main difference is in the $BH$ case where again $H$ moves $x_i$ between the  $\odot$'s but now  $b^M_\bullet$ splits it.
The $HB$ term is the same in that $\overline{D}_2$ splits $x_i$ and then $H$ moves $x_i''$ between the  $\odot$'s.  The only difference being that $H$ moves past $x_i'$ this time.
This means that the string diagrams begin from
\begin{itemize}
\item $(b^M(\mathbf{1} \cdot \mathbf{1}^{\smallboxtimes} \otimes \mathbf{1}^\otimes) \odot \mathbf{1}^\otimes \odot \mathbf{1}^{\smallboxtimes} \odot \mathbf{1}^{\otimes} ) (\mathbf{1} \cdot  \mathbf{1}^{\smallboxtimes} \odot \sigma \odot \mathbf{1}^{\smallboxtimes} \odot \mathbf{1}^\otimes)$ and
\item $(\mathbf{1} \cdot  \mathbf{1}^{\smallboxtimes} \odot \sigma \odot \mathbf{1}^{\smallboxtimes} \odot \mathbf{1}^\otimes)(\mathbf{1} \odot \mathbf{1}^\otimes \odot \mathbf{1}^{\smallboxtimes} \smallboxtimes (-\omega^{\smallboxtimes 2} \circ \overline{\Delta} \circ \sigma) \smallboxtimes \mathbf{1}^{\smallboxtimes} \odot \mathbf{1}^\otimes)$
\end{itemize}
and end at
\begin{itemize}
\item $
(b^M(\mathbf{1} \cdot \mathbf{1}^{\smallboxtimes} \otimes \mathbf{1}^\otimes) \odot \mathbf{1}^\otimes) \circ \sigma \odot \mathbf{1}^{\smallboxtimes} \odot \mathbf{1}^\otimes$, and
\item $\mathbf{1} \cdot  \mathbf{1}^{\smallboxtimes} \smallboxtimes (\omega \odot (\sigma \circ \omega)) (\overline{\Delta} \circ \sigma) \odot \smallboxtimes \mathbf{1}^{\smallboxtimes} \odot \mathbf{1}^{\otimes}
 $.
\end{itemize}
The $HB$ term has a crossing which cancels out the sign on $-\omega^{\smallboxtimes 2}$.  Luckily the sign needed to make thing cancel comes from the identity
$$
(b^M(\mathbf{1} \cdot \mathbf{1}^{\smallboxtimes} \otimes \mathbf{1}^\otimes) \odot \mathbf{1}^\otimes) \circ \sigma \odot \mathbf{1}^{\smallboxtimes} \odot \mathbf{1}^\otimes
\quad 
=
 \quad 
 -\mathbf{1} \cdot  \mathbf{1}^{\smallboxtimes} \smallboxtimes (\omega \odot (\sigma \circ \omega)) (\overline{\Delta} \circ \sigma) \odot \smallboxtimes \mathbf{1}^{\smallboxtimes} \odot \mathbf{1}^{\otimes}.
 $$

{\bf The type $\odot \sigma[\alpha'] \odot$.} 
As needed, none of these cancel.  They can only be produced by $HB$, and carry the sign coming from the operator 
$$
(\mathbf{1} \cdot  \mathbf{1}^{\smallboxtimes} \odot \sigma \odot \mathbf{1}^{\smallboxtimes} \odot \mathbf{1}^\otimes)(\mathbf{1} \odot \mathbf{1}^{\otimes} \odot \mathbf{1}^{\smallboxtimes} \smallboxtimes (-\omega^+) \odot  \mathbf{1}^{\otimes}) = 
-(\mathbf{1} \cdot  \mathbf{1}^{\smallboxtimes} \odot \mathbf{1}^{\otimes \geq 1} \odot \mathbf{1}^{\smallboxtimes} \odot \mathbf{1}^\otimes).
$$

\end{proof}


\section{Non-vanishing conditions.}
\label{section:non-vanishing}

Here we include results relevant to the non-vanishing of  $H^0(\operatorname{Mod_\infty}(A)).$
To be precise, we say  $H^0(\operatorname{Mod_\infty}(A)) = 0$
if every object is $0$, and an object $M$ is $0$ if
there is a  degree $-1$ endomorphism $G$ such that  $[B^M, G] = \mathbf{1} \odot \mathbf{1}^\otimes.$  

The basic approach to non-vanishing we take is to consider a homomorphism $S \to T$ to a commutative ring.  There is then a
functor
$$
\operatorname{Mod_\infty}(A) \to 
 \operatorname{Mod_\infty}(A \otimes T)
$$
which sends $b_\bullet \colon M \odot A[1]^\otimes \to M$ to 
$$b_\bullet \colon (M \otimes T) \odot (A\otimes T)[1]^\otimes \to M \otimes T$$
 via the isomorphism
$$
(M \otimes T) \odot (A\otimes T)[1]^\otimes \cong (M \odot A[1]^\otimes) \otimes T.
$$
If $M \otimes T$ is non-zero then so is $M$ since any contracting $G$ would 
also contract $M \otimes T$.

We have two concrete results (Theorems \ref{theorem-sufficient-vanishing} and \ref{theorem-maurer-cartan-identity-image}) that facilitate implementing this strategy.  
In short, they show that for $M \otimes T$ to have any chance of being non-zero, one must focus on the points in $\operatorname{Spec}(S)$ at which the curvature vanishes and the Maurer-Cartan function is not a submersion.
\begin{rem} {\bf (folklore)}
 Theorem \ref{theorem-sufficient-vanishing} is a sharpening of a ``folk theorem'' we learned about from L. Positselski's response to E. Segal's post on MathOverflow \cite{segal-mathoverflow}.  
Positselski refers to 
 \cite{positselski} where he recalls a conversation in which Kontsevich told him the result.
\end{rem}

\begin{rem} {\bf(critical points)}
Theorem \ref{theorem-maurer-cartan-identity-image} suggests that hiding behind the scenes, there is a version of the fact that Orlov's category of singularities \cite{orlov} depends only on a neighborhood of the singular points.
\end{rem}

\begin{prop}{\bf (sufficient non-vanishing)}
\label{prop-sufficient-nonvanishing}
Given a ring homomorphism $S \to T$ and an object $M \in 
H^0(\operatorname{Mod_\infty}(A)),$ the non-vanishing of 
$M \otimes T \in H^0(\operatorname{Mod_\infty}(A \otimes T))$
implies the
non-vanishing of $M$.
\end{prop}
\begin{proof}
Omitted.  
\end{proof}

\begin{thm}{\bf (Kontsevich-Positselski vanishing)}
\label{theorem-sufficient-vanishing}
If there is $S$-linear map $\ell \colon A \to S$ sending $\frak{m}_0(1)$ to $1$, then 
$$H^0(\operatorname{Mod_\infty}(A)) =0.$$
\end{thm}
\begin{proof}
For any module, a contracting homotopy can be built from $\ell$ and some clever algebra.
With $\lambda \colon A[1] \to S$ given by $- \ell \circ \omega$ (this carries $b_0(1) \mapsto 1 \in S$), set
$$
H = \mathbf{1} \cdot \lambda \odot \mathbf{1}^\otimes
$$
and 
$$
E = \mathbf{1} \odot \mathbf{1}^\otimes - [B,H].
$$
Then 
$$
G = (\sum_{k=0}^\infty E^k)H
$$
satisfies 
$$
\begin{array}{rcl}
[B,G] & = & [B, (\sum_{k=0}^\infty E^k)H] \\
& = &  [B,(\sum_{k=0}^\infty E^k)]H +  (\sum_{k=0}^\infty E^k)[B,H] \\
& = &  0 +  (\sum_{k=0}^\infty E^k)(\mathbf{1} \odot \mathbf{1}^\otimes -E)\\
& = & \mathbf{1} \odot \mathbf{1}^\otimes.
\end{array}
$$

It remains to verify $\sum_{k=0}^\infty E^k$ ``converges''.   Notice that 
$$
[B_0, H] = \mathbf{1} \odot \mathbf{1}^\otimes,
$$
so $E = \mathbf{1} \odot \mathbf{1}^\otimes - [B,H] = -[B_{\geq 1}, H].$  It's now easy to check that 
$$
\sum_{k \geq i} E^k 
$$
is zero on elements of on tensor degree $< i$.
\end{proof}
 
 \begin{exmp} {\bf (matrix factorizations over $S$)}
 \label{example-vanishing-matrix-factorizations}
When $(\mathcal{A}, \frak{m}_0, \cdot )$  is a curved algebra over $S$ admitting an $S$-linear map $\mathcal{A} \to S$ sending $\ell \colon \frak{m}_0(1) \to 1$, the homotopy can be expressed somewhat explicitly.  It is zero for odd indices and
$$
\gamma_{2i+2}(m \otimes f_1 \otimes \dotsm \otimes f_{2i+1}) = m \cdot \ell(f_1)\cdot L(f_2 \otimes f_3)\cdot L(f_4 \otimes f_5) \cdot \dotsm \cdot L(f_{2i} \otimes f_{2i+1})
$$
where 
$$
L(f \otimes g) = \ell(f)\cdot g + f\cdot \ell(g) + \ell(f\cdot g).
$$
 This sort of degenerate behavior is typically avoided for matrix factorizations by taking $S = \mathcal{A}$ and checking $\frak{m}_0(1)$ is not a unit.
 \end{exmp}

\begin{thm}{\bf (Maurer-Cartan differential)}
\label{theorem-maurer-cartan-identity-image}
If $A$ is uncurved (i.e. has $m_0(1) = 0$), then 
$$H^0(\operatorname{Mod_\infty}(A)) =0$$
if and only if
the unit $e \in A^{0}$ is in the image 
of the differential at $0 \in A^1$ of the Maurer-Cartan function
$$
\begin{array}{rccc}
\frak{MC} \colon & A^1& \longrightarrow &  A^2 \\
& a & \mapsto &  \sum_{k=0}^\infty   (-1)^{k \choose 2} \frak{m}_k(a^{\otimes k}).
\end{array}
$$ 
\end{thm}
\begin{proof}
For uncurved algebras, $H^0(\operatorname{Mod_\infty}(A))$ embeds faithfully into the category of unital right modules over $H^0(A)$
and  $H^0(A)$ itself is an object.  Thus we can reduce the question to whether or not 
$H^0(A) = 0.$ 

$H^0(A) = 0$ if and only if $e = m_1(a)$ for some  $a \in A^1$. The map $a \mapsto m_1(a)$ is exactly the differential of $\frak{MC}$.
\end{proof}

\begin{rem} {\bf (wishful statement)}
The statement we would like to make is 
\begin{quote}
$\text{``}H^0(\operatorname{Mod_\infty}(A)) \neq 0$
if and only if there exists a $\mathbf{k}$-valued point in $\operatorname{Spec}(S)$
for which 
$H^0(\operatorname{Mod_\infty}(A \otimes \mathbf{k})) \neq 0.\text{''}$
\end{quote}
However, further research is required to sort out under what conditions a statement like this is true.
For the time being, the results of this section allows one to show $H^0(\operatorname{Mod_\infty}(A)) \neq 0$ in some cases, and 
focus attention on the  closed set\footnote{This set is closed provided appropriate relative flatness conditions on $A^2$ and $A^2/\frak{m}_1(A^1)$ are satisfied.} of   $\mathbf{k}$-valued points at which $m_0(1) = 0$
and $e$ is not in the image of the differential of $\frak{MC}.$  
\end{rem}



\section{Recovering the classical theory}
\label{section:classical-theory}

In this section we verify that for uncurved $\mathbf{A}_\infty$-algebras over a field, an algebra map is an $\mathbf{A}_\infty$-quasi-isomorphism if an only if the resulting functors $(L,R)$ are a Quillen equivalence.
This shows that when the curvature is zero, our proposal for which morphisms constitute equivalences of $\mathbf{A}_\infty$-algebras agrees with classical notion.

\subsection{Homotopies, derivations, and inversion theorems}
This subsection is a review of some basic homotopy theory of uncurved $\mathbf{A}_\infty$-algebras and a couple results relating
The main theorem in this regard is the {\bf homotopy inversion theorem for $\mathbf{A}_\infty$-algebras} of Kadeishvili \cite{kadeishvili-85} and Prout\'e \cite{proute-86}.  We also include 
the 
{\bf homotopy inversion theorem for $\mathbf{A}_\infty$-modules}, and a proof of the fact that the {\bf adjoint algebra functor} sends $\mathbf{A}_\infty$-homotopic maps $f_\bullet$ and $f'_\bullet$ to maps $U(f)$ and $U(f')$ which are related by a $(U(f), U(f'))$-derivation.

\begin{defn}{\bf (interval coalgebra and algebra)}
The map $I_\bullet \to I_\bullet \otimes I_\bullet$
on
$$I_\bullet = \mathbf{Z} p \oplus \mathbf{Z} q \oplus \mathbf{Z} I$$ 
 sending $p \mapsto p \otimes p, q \mapsto q \otimes q$, and $I \mapsto p \otimes I + I \otimes q$
 defines a coalgebra structure on $I_\bullet$ where $\partial p =  \partial q = 0$ and $\partial I = p-q.$
We call this the {\bf interval coalgebra}.  Notice that $\mathbf{Z}q$ and $\mathbf{Z}q$ are sub-coalgebras.
We put $I$ in degree $-1$ since our codifferentials have degree 1.

Dual to the interval coalgebra is the {\bf interval algebra} 
$I^\bullet = \mathbf{Z} e_p \oplus \mathbf{Z} e_q \oplus \mathbf{Z} \epsilon$.
Just as  $\mathbf{Z}q$ and $\mathbf{Z}q$ are sub-coalgebras above, $\langle e_p, \epsilon \rangle$
and $\langle e_q, \epsilon \rangle$ are two-sided ideals with quotients isomorphic to $\mathbf{Z}$.
Again we put $\epsilon$ degree $-1$ since our differentials have degree 1.
\end{defn}

\begin{defn}{\bf ($\mathbf{A}_\infty$-homotopy)} An  {\bf $\mathbf{A}_\infty$-homotopy} between two morphisms $f^\otimes_\bullet, g^\otimes_\bullet \colon A[1]^{\otimes >0} \to A'[1]$ is a degree $1$ map
$$
h_\bullet \colon A[1]^{\otimes \geq 1} \to A'[1]
$$ 
such that the map
$$
f_\bullet \oplus g_\bullet \oplus (f_\bullet^\otimes  \otimes h_\bullet \otimes g_\bullet^\otimes) \colon  I_\bullet \otimes A[1]^\otimes \to A'[1]^\otimes
$$
is  a morphism of differential graded coalgebras (i.e. commutes with $B$ and $\Delta$).  As usual, in this case we say $f_\bullet$ and $g_\bullet$ are {\bf $\mathbf{A}_\infty$-homotopic}.
\end{defn}

\begin{defn}{\bf (homotopy)}
A morphism $\gamma$ of cocomplexes of $S$-modules is a {\bf homotopy equivalence}  if there is a morphism
$\phi$ called a {\bf homotopy inverse} in the opposite direction, such that $1 - \phi \gamma$ and $1- \gamma \phi$ are coboundaries.
A morphism $g_\bullet \colon A \to A'$ of uncurved $\mathbf{A}_\infty$-algebras is called a {\bf homotopy equivalence} if  $g_1$ is a homotopy equivalence of cocomplexes.  
\end{defn}

\begin{thm}{\bf \cite{kadeishvili-85, proute-86} (see also \cite[Theorem 4.2.44]{fooo})
 (homotopy inversion theorem for $\mathbf{A}_\infty$-algebras)}
If $g_\bullet \colon A \to A'$ is a homotopy equivalence of uncurved $\mathbf{A}_\infty$-algebras that are free as $S$-modules, then there is a homotopy equivalence $\ell_\bullet \colon A' \to A$
such that 
\begin{itemize}
\item $g_1$ and $\ell_1$ are homotopy inverses,
\item $\ell_\bullet \circ g_\bullet$ is $\mathbf{A}_\infty$-homotopic to the identity on $A$, and
\item $g_\bullet \circ \ell_\bullet$ is $\mathbf{A}_\infty$-homotopic to the identity on $A'$.
\end{itemize}  
\end{thm}

\begin{lem}{\bf ($A \sim U_e(A)$)}
\label{lemma-uncurved-equals-adjoint} 
For a strictly unital $\mathbf{A}_\infty$-algebra $A$ over a field\footnote{It is do to our ignorance of the global contracting homotopy in this proof that we must make the unfortunate requirement that  the coefficients lie in a field in the Theorem \ref{theorem:quillen=classical}.} $S$ 
with $b_0 = 0$,  $U_e(A)$ is also uncurved
$$
i_\bullet \colon A[1]^{\otimes \geq 1} \to U_e(A)[1]
$$
of Theorem \ref{theorem-universal-property-of-U_e(A)} 
is a homotopy equivalence of $\mathbf{A}_\infty$-algebras.
\end{lem}
\begin{proof}
This map is unital, so we need only check that 
$$
i_1 \colon A[1] \to U_e(A)[1]
$$
is a homotopy equivalence, and since we are over a field it suffices to check it is a quasi-isomorphism.  

After these reductions, the proof is essentially the same as that of \cite[Lemme 1.3.2.3]{lefevre-hasegawa}.
With $H$ is defined as 
$$
H(x_1 \smallboxtimes \dotsm \smallboxtimes x_k) = 
\left\{
\begin{array}{cl}
(-1)^a \ \omega[\sigma(a) \otimes \sigma(x_2)] \smallboxtimes x_3 \smallboxtimes \dotsm \smallboxtimes x_k & \text{ if $x_1 = a$ and,} \\
0 & \text{ otherwise }
\end{array}
\right.
$$
the map 
$\mathbf{1} - [d,H]$ is the identity on $A \subseteq U_e(A)$, and for any element $u$ there is an $\ell$ such that $(\mathbf{1} - [d,H])^\ell(u) \in A.$  This guarantees that for closed $u$ and $a = (\mathbf{1} - [d,H])^\ell(u)$ we have $u-a = d\hat{H}(u)$ for $\hat{H} = H\sum_{i =0}^{\ell-1} (dH)^i.$
\end{proof}

\begin{defn}{\bf ($(\phi, \psi)$-derivation)} Given two morphisms of differential graded algebras $\phi, \psi \colon \mathcal{A} \to \mathcal{A}'$, and $(\phi, \psi)$-derivation $D$ is a degree $1$ map
$$
\mathcal{A} \to \mathcal{A}'
$$
such that $D(ab) =  D(a) \psi(b) + (-1)^a \phi(a) D(b)$.  Notice this data is the same as a homomorphism 
$$
(\phi, \psi, D) \colon \mathcal{A} \to I^\bullet \otimes \mathcal{A}'
$$
sending $a \mapsto  e_0 \phi(a) + e_1 \psi(a) + \epsilon D(a)$.
  For convenience, we put $\epsilon$ degree $-1$.
\end{defn}

\begin{repthm}{theorem:homotopic-dgs-have-homotopic-modules} {\bf (homotopic algebras and modules)}
If a morphism $f \colon \mathcal{A} \to \mathcal{B}$ of differential graded algebras over $S$ is a homotopy equivalence when considered as a morphism of complexes of $S$-modules, 
then for any $\mathcal{A}$-module $M$ and $\mathcal{B}$-module $N$, the adjunction morphisms
\begin{itemize}
\item $\eta \colon M \to M \otimes_{\mathcal{A}} \mathcal{B}$ and
\item $\epsilon \colon N  \otimes_{\mathcal{A}} \mathcal{B} \to N$
\end{itemize}
are homotopy equivalences when considered as morphisms of complexes of $S$-modules.
\end{repthm}
\begin{proof}
See Appendix \ref{section-some-homological-algebra}.
\end{proof}

\begin{prop}
{\bf ($\mathbf{A}_\infty$-homotopy $\leadsto$ $U_e$-derivation)}
An $\mathbf{A}_\infty$-homotopy $h_\bullet \colon A[1]^\otimes \to A'[1]$
between morphism $f_\bullet$ and $h_\bullet$ defines a 
$(U_e(f), U_e(g))$-derivation $D \colon U_e(A) \to U_e(A')$
by $D = \mathbf{1}^{\smallboxtimes} \smallboxtimes (f_\bullet^\otimes \otimes h_\bullet \otimes g_\bullet^\otimes) \smallboxtimes \mathbf{1}^{\smallboxtimes}$.
\end{prop}
\begin{proof}
Omitted.
\end{proof}

\begin{repthm}{theorem:inversion-for-modules}
{\bf  (homotopy inversion theorem for $\mathbf{A}_\infty$-modules)}
Given a $\mathbf{A}_\infty$-morphism $\phi_\bullet \colon M \odot A[1]^{\otimes} \to N$ and an $S$-linear map
$\psi_1 \colon M  \to N$ which is a homotopy equivalence of cocomplexes of $S$-modules, then 
there exists $\psi_\bullet \colon N \odot A[1]^\otimes \to M$ such that 
\begin{itemize}
\item $\phi_1$ and $\psi_1$ are homotopy inverses,
\item $\psi_\bullet (\phi_\bullet \odot \mathbf{1}^\otimes)$ is $\mathbf{A}_\infty$-homotopic to the identity on $M$, and
\item $\phi_\bullet (\psi_\bullet \odot \mathbf{1}^\otimes)$ is $\mathbf{A}_\infty$-homotopic to the identity on $N$.
\end{itemize}
\end{repthm}
\begin{proof}
See Appendix \ref{section-some-homological-algebra}.
\end{proof}

\subsection{$\mathbf{A}_\infty$-homotopy equivalence $\Leftrightarrow$ Quillen equivalence}

\begin{thm}{\bf (Quillen = classical)}
\label{theorem:quillen=classical}
For $\mathbf{A}_\infty$-algebras over a field,
a morphism $f_\bullet \colon A[1]^{\otimes \geq 1} \to A'[1]$ is a homotopy equivalence, if and only if
the functors $(L_f, R_f)$ are a Quillen equivalence.
\end{thm}

\begin{proof}
In light of the homotopy inversion theorems, proving this assertion is a matter of verifying 
that the condition that
$$f_1 \colon A[1] \to A'[1]$$
is a homotopy equivalence of cocomplexes of $S$-modules is equivalent to the condition that both
\begin{itemize}
\item $\eta \colon M   \to M \otimes_{U_e(A)} U_e(A')$, and 
\item $\epsilon  \colon N \otimes_{U_e(A')} U_e(A) \to N$
\end{itemize}
are homotopy equivalences of cocomplexes of $S$-modules.  These two maps come from the adjunction
homomorphisms $\mathbf{1} \Rightarrow RL$ and $LR \Rightarrow \mathbf{1}$.

Under the assumption that $f_1$ is a homotopy equivalence, the
we have an $\mathbf{A}_\infty$-homotopy inverse $g_\bullet \colon A'[1]^\otimes \to A$.  Plugging in these maps 
into the functor  $U_e$ gives dg-homotopy inverse maps of the dg-algebras  $U_e(A)$ and $U_e(A').$
On one hand,
$U_e(f_\bullet)$ is a morphism of $U_e(A)-U_e(A)$-bimodules and $U(g_\bullet)$ is a $S$-linear homotopy inverse.
On the other 
$U_e(g_\bullet)$ is a morphism of $U_e(A')-U_e(A')$-bimodules and $U(f_\bullet)$ is a $S$-linear homotopy inverse.
In particular, the algebras are homotopy equivalent and so the adjunction morphisms are $S$-linear homotopy invertible.

On the other hand if the pair of maps are homotopy equivalences of cocomplexes of $S$-modules,
one can consider the module $M = U_e(A)$, and we see that $U(f_\bullet)$ is a homotopy equivalence.
This suffices to give the result because the commutative diagram
$$
\begin{tikzcd}
A \arrow{d}[swap]{i_\bullet}  \arrow{r}{f_\bullet} & A' \arrow{d}{i_\bullet} \\
U_e(A) \arrow[swap]{r}{U_e(f_\bullet)} &U_e(A') \\
\end{tikzcd}
$$
guarantees $f_\bullet$ is a homotopy equivalence. 
\end{proof}

\appendix

\section{Homotopy inversion.}
\label{section-some-homological-algebra}

Here we provide proofs of two basic facts: the {\bf adjunction morphisms associated to  a morphism of dg-algebras that is a homotopy equivalence are homotopy equivalences}, and the {\bf homotopy inversion theorem} for $\mathbf{A}_\infty$-modules.

\subsection{Homotopic dg-algebras.}

\begin{thm}{\bf (homotopic algebras and modules)}
\label{theorem:homotopic-dgs-have-homotopic-modules}
If a morphism $\mathfrc{f} \colon \mathcal{A} \to \mathcal{B}$ of differential graded algebras over $S$ is a homotopy equivalence when considered as a morphism of complexes of $S$-modules, 
then for any $\mathcal{A}$-module $M$ and $\mathcal{B}$-module $N$, the adjunction morphisms
\begin{itemize}
\item $\eta \colon M \to M \otimes_{\mathcal{A}} \mathcal{B}$ and
\item $\epsilon \colon N  \otimes_{\mathcal{A}} \mathcal{B} \to N$
\end{itemize}
are homotopy equivalences when considered as morphisms of complexes of $S$-modules.
\end{thm}
\begin{proof} 
First consider the unit morphism written as $\eta \colon M \otimes_\mathcal{A} \mathcal{A} \to M \otimes_{\mathcal{A}} \mathcal{B}$.  The main difficulty we 
must address is that the homotopy inverse  $\mathcal{B} \to \mathcal{A}$ is not necessarily $\mathcal{A}$ linear, so we cannot assume it defines a map
$M \otimes_\mathcal{A} \mathcal{B} \to M \otimes_{\mathcal{A}} \mathcal{A}$.  The solution is to move the construction up to a bar-type resolution where all tensors 
are over $S$.  

Writing $C$ for the mapping cone of $\mathfrc{f}$ 
notice that $M \otimes_\mathcal{A} C$ is the mapping cone for $\eta$.  Our problem is equivalent to finding an $S$-linear contracting homotopy for $M \otimes_\mathcal{A} C$.
Consider the bar resolution of $M \otimes_\mathcal{A} C$.  This has terms 
$$
\oplus_{k, \ell \geq 0} M \otimes \mathcal{A}^{\otimes k} \otimes C \otimes \mathcal{A}^{\otimes \ell}
$$
were all tensors are over $S$.
This is equipped with the differential $d+B$ where $B$ is the bar differential and $d$ is the differential coming from $M$, $\mathcal{A}$ and $C$.  Note $dB = Bd$,
and that this resolution is $S$-linearly homotopic to C, where the homotopy  sticks ``$\pm \otimes e$'' to the right end of a tensor.

For the moment ignore the $B$ in the differential.  Since $\mathfrc{f}$ is a $S$-linear homotopy equivalence, we know that $C$ is admits an $S$-linear contracting homotopy.  This can be promoted to an $S$-linear contracting homotopy $h$ for 
$\oplus_{k, \ell \geq 0} M \otimes \mathcal{A}^{\otimes k} \otimes C  \otimes \mathcal{A}^{\otimes \ell}$ equipped only with the differential $d$.  

Finally, we can promote $h$ to a contracting homotopy $H$ for $d+B$ by setting
$$
H = h \sum_{i =0}^\infty (-Bh)^i.
$$
This is well defined since for any element $x$ in the resolution there is an $j$ such that $(-Bh)^jx=0$.  We leave it to the reader to verify that $\mathbf{1} = dh+hd$ implies $\mathbf{1} = (d+B)H+H(d+B)$ under the assumption $dB=Bd$.

Now for the counit morphism $\epsilon \colon N  \otimes_{\mathcal{A}} \mathcal{B} \to N$.  In this case we do have a well-defined map back $\sigma \colon N \to N \otimes_\mathcal{A} \mathcal{B}$
sending $n \mapsto n \otimes e$.  The composition $\epsilon \sigma$ is the identity on $N$, so $\epsilon$ is a right inverse to $\sigma$.  We will show that $\sigma$ is a homotopy invertible, which proves that 
$\epsilon$ is also homotopy invertible. 

Observe that under the identification $N = N \otimes_\mathcal{A} \mathcal{A}$, the map $\sigma$ becomes the map 
$$N \otimes_\mathcal{A} \mathcal{A} \to N \otimes_\mathcal{A} \mathcal{B}$$
induced by $\mathfrc{f}$.  We can now run the argument used for the unit morphism to produce an $S$-linear homotopy inverse.
\end{proof}

\subsection{The homotopy inversion theorem for $\mathbf{A}_\infty$-modules}

We provide a proof of the homotopy inversion theorem for modules.  This proof follows closely the proof for algebras in \cite{fooo}.

\begin{rem}
The well-definedness of most operations below is a consequence of the fact
that the spaces $\operatorname{Hom}_S(M \otimes A[1]^{\otimes \geq k+1}, N)$
form an {\bf ideal}.  By this it is meant that they are closed under pre- and post-composition, and under the operator
$[B, -]$.
\end{rem}

\subsubsection{The obstruction complex.}

The obstruction complex serves two roles.  The first is to determine if a given $A_k$-morphism can be extended to an $A_{k+1}$-morphism, and the second is to determine if an homotopy between two unobstructed  $A_k$-morphisms can be extended along with them.

To be clear, we say something can be {\bf extended} from $A_k$ to $A_{k+1}$ if there is an $A_{k+1}$
object that agrees with the thing in question on all tensors of degree $k$ and below.  In addition,
we use the notation
$$[B, \phi_\bullet] = b_\bullet \circ (\phi_\bullet \odot \mathbf{1}^{\otimes}) -(-1)^\phi \phi_\bullet \circ B.$$

\begin{defn}{\bf (homomorphisms and homotopies)}
An {\bf $A_k$-morphism} is a degree 0 map
$$
\phi_\bullet \colon M \odot A[1]^{\otimes} \to N
$$
such that
$$
[B, \phi_\bullet] 
= \delta_\bullet
\colon M \odot A[1]^{\otimes} \to N 
$$
lies in $\operatorname{Hom}_S(M \odot A[1]^{\otimes \geq k}, N)$.
A {\bf homotopy} is an arbitrary degree -1 map
$$
h_\bullet \colon M \odot A[1]^{\otimes} \to N, 
$$
and two morphisms are said to be $A_k$-homotopic via $h_\bullet$
if
$$
 \phi_\bullet - \psi_\bullet - [B, h_\bullet] \in \operatorname{Hom}_S(M \odot A[1]^{\otimes \geq k}, N).
$$
\end{defn}

\begin{defn}{\bf (obstruction complex)}
$$
C_{\text{Obs}}(k+1)(M,N)  \ = \ \operatorname{Hom}_S(M \odot A[1]^{\otimes \geq k}, N)  \ / \
\operatorname{Hom}_S(M \odot A[1]^{\otimes \geq k+1}, N)
$$
with the codifferential $[B,-]$. The cohomology is  denoted
$$
H_{\text{Obs}}(k+1)(M,N).
$$
\end{defn}

\begin{defn}{\bf (obstruction class of a morphism)}
The  obstruction class  of an $A_k$-morphism $\phi_\bullet$ is the class in 
$$[\operatorname{obs}(\phi_\bullet)] = [ \  [B, \phi_\bullet] \ ] \in H_{\text{Obs}}(k+1)(M,N).$$
Presented this way, it is easy to see it is closed as
$$
[B, \operatorname{obs}(\phi_\bullet)] = [B,[B,\psi_\bullet]].
$$
\end{defn}

\begin{defn}{\bf (obstruction of a homotopy)}
In addition if $h_\bullet$ is a homotopy between $A_k$-morphisms $\phi_\bullet$  and $\psi_\bullet$:
$$
  \phi_\bullet - \psi_\bullet - [B, h_\bullet] \in \operatorname{Hom}_S(M \odot A[1]^{\otimes \geq k}, N)
 $$
 it also defines an obstruction
 $$\operatorname{obs}(\phi_\bullet, \psi_\bullet, h_\bullet) =   \phi_\bullet - \psi_\bullet - [B, h_\bullet] \in C_{\text{Obs}}(k+1)(M,N). $$
 However, this element \emph{might not be closed}.
\end{defn}

\begin{prop} {\bf (unobstructed morphisms extend)}
An $A_k$-morphism $\phi_\bullet \colon M \otimes A[1]^\otimes \to N$, can be extended to an $A_{k+1}$-morphism
if and only if $$0 = [\operatorname{obs}(\phi_\bullet)] \in H_{\text{Obs}}(k+1)(M,N).$$
\end{prop}
\begin{proof}
Extend $\phi_\bullet$ to $\phi_\bullet + X_{k+1}$ for an unknown $X_{k+1} \colon M \odot A[1]^{\otimes k} \to N.$
  The condition that $\phi_\bullet + X_{k+1}$ is an $A_{k+1}$ morphism is exactly that
  $[B,\phi_\bullet + X_{k+1}] = 0$ in  $C_{\text{Obs}}(k+1)(M,N)$.  In other words,
  $[B,-X_{k+1}]  = \operatorname{obs}(\phi_\bullet)$ in $C_{\text{Obs}}(k+1)(M,N)$.
\end{proof}

\begin{prop}{\bf (homotopic morphisms have homotopic obstructions)}
If  an $A_k$-morphism 
 $\phi_\bullet
 \colon M \odot A[1]^{\otimes k-1} \to N$ and 
 $\psi_\bullet
 \colon M  \odot A[1]^{\otimes k-1} \to N $ 
are related by a homotopy $h_\bullet$:
 $$
  \phi_\bullet - \psi_\bullet - [B, h_\bullet] \in \operatorname{Hom}_S(M \odot A[1]^{\otimes \geq k}, N)
 $$
then their obstruction classes are homotopic.
\end{prop}
\begin{proof}
 The obstruction class associated with the triple provides the homotopy.  Indeed,  
$[B, \operatorname{obs}(\phi_\bullet, \psi_\bullet, h_\bullet)]$ is simply  
$$[B,  \phi_\bullet - \psi_\bullet - [B, h_\bullet]] = 
\operatorname{obs}(\phi_\bullet)
-
\operatorname{obs}(\psi_\bullet)
$$
when considered in the obstruction complex.
\end{proof}

\begin{prop}{\bf (obstructions of homotopies between unobstructed morphisms are closed)}
If $\phi_\bullet$ and $\psi_\bullet$ are unobstructed $A_{k}$-morphisms
and $h_\bullet$ is a homotopy between them 
$$
  \phi_\bullet - \psi_\bullet - [B, h_\bullet] \in \operatorname{Hom}_S(M \odot A[1]^{\otimes \geq k}, N)
$$
then 
$\operatorname{obs}(\phi_\bullet, \psi_\bullet, h_\bullet) \in C_{\text{Obs}}(k+1)(M,N)$ is closed.
\end{prop}
\begin{proof}
Omitted.
\end{proof}

\begin{prop} {\bf (homotopies with exact obstructions extend)}
If $A_{k+1}$-morphisms
 $\phi_\bullet
 \colon M  \odot A[1]^{\otimes } \to N$ and 
 $\psi_\bullet
 \colon M  \odot A[1]^{\otimes } \to N $ 
are 
 related by an homotopy $h_\bullet$ as $A_k$-morphisms:
 $$
  \phi_\bullet - \psi_\bullet - [B, h_\bullet] \in \operatorname{Hom}_S(M \odot A[1]^{\otimes \geq k}, N)
 $$
 then we know $\operatorname{obs}(\phi_\bullet, \psi_\bullet, h_\bullet)$
is closed and
 $h_\bullet$ can be extended to a homotopy between 
 $\phi_\bullet$ and $\psi_\bullet$ as $A_{k+1}$-morphisms
 if and only if   
 $$
0 = [ \operatorname{obs}(\phi_\bullet, \psi_\bullet, h_\bullet) ] \in 
  H_{\text{Obs}}(k+1)(M,N).
$$
\end{prop}
\begin{proof}
In order to have an extension of the homotopy we must solve for $X_{k+1}$ in the  equation
$$
\phi_\bullet - \psi_\bullet = [B, h_\bullet + X_{k+1}]
$$
in $C_{\text{Obs}}(k+1)(M,N)$.
\end{proof}

\subsubsection{Computational aspects.}

There are two convenient computational facts relevant to determining if an obstruction class is zero.  The first is that  the obstruction complex has the structure of a dg-bimodule, and the second is that the formation of obstructions
acts as a derivation for which the $A_{k+1}$-morphisms are ``constants''.

\begin{prop}{\bf ($\operatorname{Mod}_\infty(A)$-action)}
$\operatorname{Mod}_\infty(A)$ acts on $C_{\text{Obs}}(k+1)$ as to make it a dg-bimodule.  That is to say with 
$\phi_\bullet \colon L \colon A[1]^\otimes \to M$, $\psi_\bullet \colon N \odot A[1]^\otimes \to P$ and $c_\bullet \in C_{\text{Obs}}(k+1)(M,N)$, we have
$$
\phi_\bullet \circ c_\bullet \circ \psi_\bullet \in C_{\text{Obs}}(k+1)(L,P)
$$
and 
$$
[B, \phi_\bullet \circ c_\bullet \circ \psi_\bullet] = \ 
[B, \phi_\bullet ] \circ c_\bullet \circ \psi_\bullet  \ +  \ 
(-1)^\phi \phi_\bullet \circ [B, c_\bullet ] \circ \psi_\bullet \ + \ 
(-1)^{\phi}(-1)^c \phi_\bullet \circ c_\bullet \circ [B, \psi_\bullet].
$$
Furthermore, any
$
\phi_\bullet \in \operatorname{Hom}_S(L \odot A^{\otimes \geq 1}, M) 
$
and 
$
\psi_\bullet \in \operatorname{Hom}_S(N \otimes A^{\otimes \geq 1}, P)
$
act trivially.
\end{prop}
\begin{proof}
Omitted.
\end{proof}

\begin{prop}{\bf (obstruction derivation)}
The assignment of an $A_k$-morphism to its obstruction class is a derivation for which $A_{k+1}$-morphisms act as constants.  To be clear, if $\phi_\bullet$ and $\psi_\bullet$ are $A_k$-morphisms and both
$\alpha_\bullet$ and $\beta_\bullet$ are $A_{k+1}$-morphisms, then
$$
[B, \alpha_\bullet \circ \phi_\bullet \circ \psi_\bullet \circ \beta_\bullet] = \alpha_\bullet \circ ([B, \phi_\bullet] \circ \psi_\bullet +  (-1)^\phi \phi_\bullet \circ [B, \psi_\bullet]) \circ \beta_\bullet
$$
\end{prop}
\begin{proof}
Omitted.
\end{proof}

\subsubsection{Homotopy inversion theorem for $\mathbf{A}_\infty$-modules}

The main application the homotopy inversion theorem for  $\mathbf{A}_\infty$-modules is the fact that when an $\mathbf{A}_\infty$-algebra are uncurved, a morphism $\phi_\bullet$ of modules descends to a isomorphism in $\operatorname{H}^0(\operatorname{Mod}_\infty(A))$
if and only if $\phi_1$ is a homotopy equivalence of complexes of $S$-modules.

There are two main tricks exploited in the proof.  The first is that, by virtue of being an $\mathbf{A}_\infty$-morphism, $\phi_\bullet \circ -$
induces cochain homotopy equivalences of obstruction complexes.  This is used twice: once to argue the unobstructedness of $\phi_\bullet$, and second to contrive
a ``correction'' $\hat{\psi}_\bullet$ to the na\"ive extension $\tilde{\psi}_\bullet$ of $\psi_\bullet$ in such a way that the  $A_k$-homotopy $h_\bullet$ extends to an $A_{k+1}$-homotopy 
between  $\mathbf{1}$ and $\phi_\bullet(\hat{\psi}_\bullet \odot \mathbf{1}^\otimes)$.

\begin{thm}{\bf homotopy inversion theorem for $\mathbf{A}_\infty$-modules}
\label{theorem:inversion-for-modules}
Given an $\mathbf{A}_\infty$-morphism $\phi_\bullet \colon M \odot A[1]^{\otimes} \to N$, an $A_k$-morphism
$\psi_\bullet \colon N \odot A[1]^{\otimes } \to M$ and an $A_k$-homotopy
$h_\bullet$: 
$$
\mathbf{1} - \phi_\bullet (\psi_\bullet \odot \mathbf{1}^\otimes) - [B,h_\bullet] \in  \operatorname{Hom}_S(N \odot A[1]^{\otimes \geq k}, N)
$$
and an $S$-linear map of degree $-1$
$$
\ell \colon M \to M
$$
such that 
$$
\mathbf{1} - \psi_\bullet (\phi_\bullet \odot \mathbf{1}^{\otimes}) - [B,\ell] \in  \operatorname{Hom}_S(M \odot A[1]^{\otimes \geq 1}, M)
$$
then there are $\mathbf{A}_\infty$-extensions  $\hat{\psi}_\bullet$ and $\hat{h}_\bullet$  such that
$$
\mathbf{1} - \phi_\bullet (\hat{\psi}_\bullet \odot \mathbf{1}^\otimes) = [B,\hat{h}_\bullet] .
$$
\end{thm}
\begin{proof}
It suffices to show that we can extend these to $A_{k+1}$-objects.

The map $\phi_\bullet \circ - \colon C_{ob}(k+1)(N,M) \to C_{ob}(k+1)(N,N)$
is a homotopy equivalence and commutes with the formation of obstruction classes of morphisms.  
Commutation follows from the fact that $\phi_\bullet$ is $\mathbf{A}_\infty$ and thus a fortiori $A_{k+1}$.
For the homotopy equivalence, the inverse is 
$$
\psi_\bullet \circ - \colon C_{ob}(k+1)(N,N) \to C_{ob}(k+1)(N,M),
$$
and the equations 
$$
\mathbf{1} - \phi_\bullet (\psi_\bullet \odot \mathbf{1}^\otimes) - [B,h_\bullet] \equiv 0
$$
and
$$
\mathbf{1} - \psi_\bullet (\phi_\bullet \odot \mathbf{1}^{\otimes}) - [B,\ell] \equiv 0.
$$
give the homotopies.

Consequently $\psi_\bullet$  extends to an $A_{k+1}$-morphism $\tilde{\psi}_\bullet$.
 The extension of $\psi_\bullet$ is guaranteed because the cochain homotopy equivalence $\phi_\bullet \circ -$
carries it to  $\phi_\bullet \circ (\psi_\bullet \odot \mathbf{1}^\otimes)$ and this is homotopic to $\mathbf{1}$.  The obstruction class of 
$\mathbf{1}$ is zero and so all of these objects extend.

Unfortunately, the  $A_k$-homotopy $h_\bullet$ between $\mathbf{1}$ and the composition $\phi_\bullet (\tilde{\psi}_\bullet \odot \mathbf{1}^\otimes)$
might not extend. However, it is closed and the trick is to add a term $X_{k+1}$ to $\tilde{\psi}_\bullet$ so that 
$\phi_\bullet ((\tilde{\psi}_\bullet - X_{k+1}) \odot \mathbf{1}^\otimes)$ is $A_{k+1}$-homotopic to $\mathbf{1}$.  To do this,
consider $X_{k+1} \colon N \odot A[1]^{\otimes \geq k} \to M$ such that 
$$
[\phi_\bullet (X_{k+1} \odot \mathbf{1}^\otimes) ] = [ \
\mathbf{1} - \phi_\bullet (\tilde{\psi}_\bullet \odot \mathbf{1}^\otimes) - [B, h_\bullet]
\ ] .
$$
This means there is an odd map 
$Y_{k+1} \colon N \odot A[1]^{\otimes \geq k} \to N$
such that 
$$
\mathbf{1} - \phi_\bullet (\tilde{\psi}_\bullet \odot \mathbf{1}^\otimes) - [B, h_\bullet] - \phi_\bullet (X_{k+1} \odot \mathbf{1}^\otimes) = 
[B, Y_{k+1}]
$$
in $C_{ob}(k+1)(N,N)$.   Finally, rearranging terms gives us what we need:
$$
\mathbf{1} - \phi_\bullet ((\tilde{\psi}_\bullet - X_{k+1}) \odot \mathbf{1}^\otimes) - [B, h_\bullet+Y_{k+1}] \in \operatorname{Hom}_S(N \odot A[1]^{\otimes \geq k+1} , N)
$$
with our extensions $\hat{\psi}_\bullet = \tilde{\psi}_\bullet - X_{k+1}$ and $\hat{h}_\bullet = h_\bullet+Y_{k+1}.$
\end{proof}

\section{Homotopy theory.}
\label{section-homotopy-and-consequences}

The results here allow for a homotopical treatment of 
$H^0(\operatorname{Mod_\infty}(A))$.  This
leads to a sophisticated approach to the notion of a homotopy equivalence of $\mathbf{A}_\infty$-algebras:
i.e. morphisms $f$ for which the Quillen adjunction $(L_f, R_f)$ below is a Quillen equivalence.
We describe this here 
essentially without proof, since the work required 
to verify the statements below is routine and follows standard
constructions.  We 
 refer the interested reader to Armstrong \cite{armstrong}.

{\bf Basic statement.} The subcategory of strict morphisms can be considered as a functor 
$$
\operatorname{Mod_\infty^{\text{st}}} \colon \mathbf{Alg}_{\infty} \to \mathbf{ModCat}
$$
from the category 
of strictly unital curved $\mathbf{A}_\infty$-algebras to the category of Quillen model categories and Quillen adjunctions.  This functor comes with an equivalence 
$$
\operatorname{Mod_\infty^{\text{st}}}(A)[\mathcal{W}_A^{-1}] \xrightarrow{\sim}  H^0(\operatorname{Mod_\infty}(A))
$$
from the homotopy category of the strict morphisms 
to the $0^{\text{th}}$-cohomology category of  $\operatorname{Mod_\infty}(A).$

{\bf Homotopy theory.} The output of the functor $\operatorname{Mod_\infty^{\text{st}}}$  evaluated on an $\mathbf{A}_\infty$-algebra $A$ is 
\begin{itemize}
\item the subcategory $\operatorname{Mod_\infty^{\text{st}}}(A) \subseteq 
\operatorname{Mod_\infty}(A)$ of strict morphisms  between right $A$-modules,
\item a subcategory of weak equivalences $\mathcal{W}_A \subseteq \operatorname{Mod_\infty^{\text{st}}}(A)$ made up of those morphisms whose image under functor
$$
\operatorname{Mod_\infty^{\text{st}}}(A) \to  H^0(\operatorname{Mod_\infty}(A))
$$
is an isomorphism, 
\item a subcategory of fibrations $\mathcal{F}_A \subseteq \operatorname{Mod_\infty^{\text{st}}}(A)$ 
made up of morphisms which are split surjections of $S$-modules, and
\item the  functor
$$
Q_A \colon \operatorname{Mod_\infty^{\text{st}}}(A) \to \operatorname{Mod_\infty^{\text{st}}}(A)
$$
with the natural transformation $\epsilon \colon Q_A \Rightarrow \mathbf{1}$  (defined in Theorem \ref{theorem-definition-of-Q}).
\end{itemize}
This data defines a model structure on $\operatorname{Mod_\infty^{\text{st}}}(A)$ for which all objects are fibrant and $Q_A$ is a cofibrant replacement functor.

{\bf Functoriality.} To a morphism of algebras $f_\bullet \colon A \to A'$ the functor assigns
\begin{itemize}
\item a restriction of scalars functor (Proposition \ref{proposition-A-infinity-restriction-of-scalars}) 
$$R_f \colon \operatorname{Mod_\infty^{\text{st}}}(A') \to \operatorname{Mod_\infty^{\text{st}}}(A),$$
\item which admits a left adjoint (Definition \ref{proposition-adjoint-to-restriction-of-scalars})
$$L_f \colon \operatorname{Mod_\infty^{\text{st}}}(A) \to \operatorname{Mod_\infty^{\text{st}}}(A').$$
\end{itemize}
Together $(L_f, R_f)$ is a Quillen adjunction\footnote{In addition, these functors are triangulated.} and the assignment $f_\bullet \mapsto 
(L_f, R_f)$ is functorial.

{\bf Compatibility.} 
When the $\mathbf{A}_\infty$-algebra is a curved dg-algebra $\mathcal{A}$, its category of dg-modules $\operatorname{Mod}_\text{dg}^\text{st}(\mathcal{A})$
can be equipped with the model structure in two ways: one inherited from that of $\operatorname{Mod}_\infty^\text{st}(\mathcal{A})$, and the other its   ``usual'' model structure.  Fortunately these coincide.
The ``usual''
model structure on $\operatorname{Mod}_\text{dg}^\text{st}(\mathcal{A})$ is the one in which
\begin{itemize}
\item split surjections $\mathcal{F}'_\mathcal{A}$ are fibrations,
\item the subcategory
$\mathcal{W}'_\mathcal{A} \subseteq 
\operatorname{Mod}_\text{dg}^\text{st}(\mathcal{A})
$
of weak equivalences is taken
to be those morphisms which become quasi-isomorphisms under $\mathcal{Y} \circ \, Q_\mathcal{A}'$ where
$$
\mathcal{Y}(M) = \operatorname{Mod}_\text{dg}(\mathcal{A})( - ,M) 
$$
is the Yoneda embedding, and
\item the bar resolution $Q_\mathcal{A}' = - \overset{\infty}{\otimes}_{\mathcal{A}} \ \mathcal{A}[1]$
as a cofibrant replacement functor\footnote{Notice that in this case $M \overset{\infty}{\otimes}_{\mathcal{A}} \  \mathcal{A}[1] = (M[1] \odot \mathcal{A}[1]^{\otimes \geq 1})[-1]$ is the bar resolution up to signs on the codifferential.}.
\end{itemize}
On the subcategories where both model structures are defined, they agree: 
 \begin{itemize}
 \item $\operatorname{Mod}_\text{dg}^\text{st}(\mathcal{A}) \subseteq \operatorname{Mod_\infty^{\text{st}}}(\mathcal{A})$ is a full subcategory,
  \item 
  $\mathcal{F}'_\mathcal{A} = \mathcal{F}_\mathcal{A} \cap \operatorname{Mod}_\text{dg}^\text{st}(\mathcal{A})$ and
  $\mathcal{W}'_\mathcal{A} = \mathcal{W}_\mathcal{A} \cap \operatorname{Mod}_\text{dg}^\text{st}(\mathcal{A}),$ 
 \item the image of $Q_\mathcal{A}$ restricted to $\operatorname{Mod}_\text{dg}^\text{st}(\mathcal{A})$ lies in $\operatorname{Mod}_\text{dg}^\text{st}(\mathcal{A}),$ 
  \item $Q_\mathcal{A}'$ and $Q_\mathcal{A}$ are cofibrant replacement functors for both model structures, and 
\item the path space for any object in $\operatorname{Mod}_\text{dg}^\text{st}(\mathcal{A})$ 
 lies in $\operatorname{Mod}_\text{dg}^\text{st}(\mathcal{A}).$  
 \end{itemize}

When the curved dg-algebra in question is $\mathcal{A}=U_e(A)$ 
there are two model category structures on
$\operatorname{Mod}_\text{dg}^\text{st}(U_e(A))$:
one in which we forget that $\mathcal{A}$ came from $A$, and another  
coming from the identification with  
$\operatorname{Mod}_\infty^\text{st}(A)$.  These too are compatible.  In particular, 
\begin{itemize}
\item $\operatorname{Mod_\infty^{\text{st}}}(A) = \operatorname{Mod_\text{dg}^{\text{st}}}(\mathcal{A}),$ 
 \item $\mathcal{W}'_\mathcal{A} = \mathcal{W}_A $, and
 \item $\mathcal{F}'_\mathcal{A} = \mathcal{F}_A$.
\end{itemize}
This means $Q_A$, $Q'_\mathcal{A},$ and $Q_\mathcal{A}$ are all cofibrant replacement functors on $\operatorname{Mod}_\text{dg}^\text{st}(\mathcal{A})$ etc.

\begin{exmp}{\bf (matrix factorizations)}
\label{example-matrix-factorizations}
Given a commutative ring with identity $A$ and an element $W \in A$, one can consider $(A, 0, W)$
as a curved dg-algebra.  Dg-modules over $A$ which are finite rank and free as $A$-modules were first 
considered by Eisenbud under the name \emph{matrix factorizations} \cite{eisenbud}. The statements above imply the homotopy 
category of matrix factorizations includes as a full subcategory\footnote{Strictly speaking, at the dg-level matrix factorizations include as summand whose complement is contractible.} of $H^0(\operatorname{Mod_\infty}(A))$, where the coefficient ring $S$
equals $A$ itself.
\end{exmp}

\subsection{Restriction and extension of scalars.}
\label{subsection_restriction_and_extension_of_scalars}
In order to make the statements above precise, if not proved\footnote{We refer to \cite{armstrong} for complete proofs.}, we define the functors used in the Quillen adjunction assigned to a morphism of $\mathbf{A}_\infty$- algebras.

\begin{prop}{\bf ($\mathbf{A}_\infty$-restriction of scalars)}
\label{proposition-A-infinity-restriction-of-scalars}
Given a morphism $f_\bullet \colon A \to A'$ of $\mathbf{A}_\infty$-algebras, we have a dg-functor
$$
R_f \colon \operatorname{Mod}_\infty(A') \to \operatorname{Mod}_\infty(A)
$$
given by sending $(M, b^M_\bullet) \mapsto (M, b^M_\bullet(\mathbf{1} \odot (f_\bullet)^\otimes))$
and $\phi_\bullet \colon M \to M'$ to $\phi_\bullet(\mathbf{1} \odot (f_\bullet)^\otimes)$.
\end{prop}
\begin{proof}
The diagram 
$$
\begin{tikzcd}
M \odot (A[1])^\otimes \arrow{d}{B} \arrow{r}{\mathbf{1} \odot (f_\bullet)^\otimes} & M \odot(A'[1])^\otimes  \arrow{d}{B}\\
M \odot (A[1])^\otimes \arrow{d}{\phi_\bullet \odot \mathbf{1}^\otimes} \arrow{r}{\mathbf{1} \odot (f_\bullet)^\otimes} & M \odot (A'[1])^\otimes  \arrow{d}{\phi_\bullet(\mathbf{1} \odot (f_\bullet)^\otimes) \odot \mathbf{1}^\otimes}\\
M' \odot (A[1])^\otimes \arrow{d}{B} \arrow{r}{\mathbf{1} \odot (f_\bullet)^\otimes} & M' \odot(A'[1])^\otimes  \arrow{d}{B}\\
M' \odot(A[1])^\otimes  \arrow{r}{\mathbf{1} \odot(f_\bullet)^\otimes} & M' \odot (A'[1])^\otimes
\end{tikzcd}
$$
commutes.
So $[B, R_f(\phi_\bullet) \odot \mathbf{1}^\otimes] = R_f(\delta\phi_\bullet) \odot \mathbf{1}^\otimes$ as needed, and it is easy to see that composition commutes with $R_f:$\\
$
\psi_\bullet( \phi_\bullet \odot \mathbf{1}^\otimes) \odot \mathbf{1}^\otimes \mapsto 
$ 
$$ \psi_\bullet( \phi_\bullet(\mathbf{1} \odot (f_\bullet)^\otimes) \odot (f_\bullet)^\otimes) \odot \mathbf{1}^\otimes
= \psi_\bullet(\mathbf{1} \odot (f_\bullet)^\otimes) \odot \mathbf{1}^\otimes \  
\circ 
 \ \   \phi_\bullet(\mathbf{1} \odot (f_\bullet)^\otimes) \odot \mathbf{1}^\otimes .
$$
\end{proof}

\begin{prop}{\bf (extension of scalars)}
\label{proposition-adjoint-to-restriction-of-scalars}
Again, given a morphism $f_\bullet \colon A \to A'$ of $\mathbf{A}_\infty$-algebras,
$R_f$ carries strict morphisms to strict morphisms, and thus can be
considered as a functor $\operatorname{Mod}^\text{st}_\infty(A') \to \operatorname{Mod}^\text{st}_\infty(A).$ 
Under the identifications 
$$\operatorname{Mod}^\text{st}_\infty(A') = \operatorname{Mod}_\text{dg}^\text{st}(U_e(A')) \  \ \text{ and } \  \ \operatorname{Mod}^\text{st}_\infty(A) = \operatorname{Mod}_\text{dg}^\text{st}(U_e(A))$$ 
is exactly the usual restriction of scalars functor, and thus admits
extension of scalars as  a left adjoint
$$
L_f \colon \operatorname{Mod}^\text{st}_\infty(A) \to \operatorname{Mod}^\text{st}_\infty(A')
$$
with 
$$
L_f(M) = M \otimes_{U_e(A)} U_e(A')
$$
\end{prop}
\begin{proof}
$R_f$ is a dg-functor, so it carries closed morphisms to closed morphisms.  Furthermore, linear morphisms ($\phi_\bullet = \phi_1$) remain linear.
\end{proof}

\newpage

 {\small
 
\bibliography{curved_A-infinity_Adj_and_Htpy}

\begin{thebibliography}{FOOO09}

\bibitem[Arm15]{armstrong}
Jeffrey Armstrong.
\newblock PhD thesis, Drexel University, 2015.

\bibitem[Eis80]{eisenbud}
David Eisenbud.
\newblock Homological algebra on a complete intersection, with an application
  to group representations.
\newblock {\em Trans. Amer. Math. Soc.}, 260(1):35--64, 1980.

\bibitem[EML53]{eilenberg-maclane-bar-construction}
Samuel Eilenberg and Saunders Mac~Lane.
\newblock On the groups of {$H(\Pi,n)$}. {I}.
\newblock {\em Ann. of Math. (2)}, 58:55--106, 1953.

\bibitem[FOOO09]{fooo}
Kenji Fukaya, Yong-Geun Oh, Hiroshi Ohta, and Kaoru Ono.
\newblock {\em Lagrangian intersection {F}loer theory: anomaly and
  obstruction.}, volume~46 of {\em AMS/IP Studies in Advanced Mathematics}.
\newblock American Mathematical Society, Providence, RI; International Press,
  Somerville, MA, 2009.

\bibitem[Ger63]{gerstenhaber-63}
M.~Gerstenhaber.
\newblock The cohomology structure of an associative ring.
\newblock {\em Ann. Math.}, 78:267--288, 1963.

\bibitem[Kad80]{kadeishvili-80}
T.~V. Kadei{\v{s}}vili.
\newblock On the theory of homology of fiber spaces.
\newblock {\em Uspekhi Mat. Nauk}, 35(3(213)):183--188, 1980.
\newblock International Topology Conference (Moscow State Univ., Moscow, 1979).

\bibitem[Kad85]{kadeishvili-85}
T.~V. Kadeishvili.
\newblock The category of differential coalgebras and the category of
  {$A(\infty)$}-algebras.
\newblock {\em Trudy Tbiliss. Mat. Inst. Razmadze Akad. Nauk Gruzin. SSR},
  77:50--70, 1985.

\bibitem[KLN10]{keller-lowen-nicolas}
Bernhard Keller, Wendy Lowen, and Pedro Nicol{\'a}s.
\newblock On the (non)vanishing of some ``derived'' categories of curved dg
  algebras.
\newblock {\em J. Pure Appl. Algebra}, 214(7):1271--1284, 2010.

\bibitem[KS01]{kontsevich-soibelman-01}
Maxim Kontsevich and Yan Soibelman.
\newblock Homological mirror symmetry and torus fibrations.
\newblock In {\em Symplectic geometry and mirror symmetry ({S}eoul, 2000)},
  pages 203--263. World Sci. Publ., River Edge, NJ, 2001.

\bibitem[KS09]{kontsevich-soibelman}
M.~Kontsevich and Y.~Soibelman.
\newblock Notes on {$A_\infty$}-algebras, {$A_\infty$}-categories and
  non-commutative geometry.
\newblock In {\em Homological mirror symmetry}, volume 757 of {\em Lecture
  Notes in Phys.}, pages 153--219. Springer, Berlin, 2009.

\bibitem[LH03]{lefevre-hasegawa}
Kenji Lef\`evre-Hasegawa.
\newblock {\it {Sur les $A_\infty$-cat\'egories}}.
\newblock Th\'ese de doctorat, Universit\'e Denis Diderot - Paris 7, November
  2003, arXiv:math/0310337.

\bibitem[LVdB06]{lowen-van-den-bergh-04}
Wendy Lowen and Michel Van~den Bergh.
\newblock Deformation theory of abelian categories.
\newblock {\em Trans. Amer. Math. Soc.}, 358(12):5441--5483 (electronic), 2006.

\bibitem[Mer99]{merkulov-99}
S.~A. Merkulov.
\newblock Strong homotopy algebras of a {K}\"ahler manifold.
\newblock {\em Internat. Math. Res. Notices}, (3):153--164, 1999.

\bibitem[NZ13]{nikolov-zahariev-13}
Nikolay~M. Nikolov and Svetoslav Zahariev.
\newblock Curved {$A_\infty$}-algebras and {C}hern classes.
\newblock {\em Pure Appl. Math. Q.}, 9(2):333--369, 2013.

\bibitem[Orl04]{orlov}
D.~O. Orlov.
\newblock Triangulated categories of singularities and {D}-branes in
  {L}andau-{G}inzburg models.
\newblock {\em Tr. Mat. Inst. Steklova}, 246(Algebr. Geom. Metody, Svyazi i
  Prilozh.):240--262, 2004.
\newblock {Translation in Proc. Steklov Inst. Math. 2004, no. 3 (246),
  227--248.}

\bibitem[Pos11]{positselski}
Leonid Positselski.
\newblock Two kinds of derived categories, {K}oszul duality, and
  comodule-contramodule correspondence.
\newblock {\em Mem. Amer. Math. Soc.}, 212(996):vi+133, 2011.

\bibitem[Pos12]{positselski-weak}
Leonid Positselski.
\newblock {Weakly curved A-infinity algebras over a topological local ring},
  2012, {\it arXiv:1202.2697}.

\bibitem[Pro86]{proute-86}
Alain Prout\'e.
\newblock {\it {Alg\`ebres diff\`erentielles fortement homotopiquement
  associatives ($A_\infty$-alg\`ebres)}}.
\newblock Th\'ese de doctorat, Universit\'e Paris 7, 1986.

\bibitem[PS95]{penkava-schwarz}
Michael Penkava and Albert Schwarz.
\newblock {$A_\infty$} algebras and the cohomology of moduli spaces.
\newblock In {\em Lie groups and {L}ie algebras: {E}. {B}. {D}ynkin's
  {S}eminar}, volume 169 of {\em Amer. Math. Soc. Transl. Ser. 2}, pages
  91--107. Amer. Math. Soc., Providence, RI, 1995.

\bibitem[Seg]{segal-mathoverflow}
Ed~Segal.
\newblock {Homotopic morphisms between curved A-infinity algebras}.
\newblock MathOverflow, http://mathoverflow.net/q/86821.

\bibitem[Sta63]{stasheff-published-thesis}
James~Dillon Stasheff.
\newblock Homotopy associativity of {$H$}-spaces. {I}, {II}.
\newblock {\em Trans. Amer. Math. Soc. 108 (1963), 275-292; ibid.},
  108:293--312, 1963.

\end{thebibliography}
\bibliographystyle{halpha}  

\noindent
{Department of Mathematics, Drexel University, Philadelphia, PA 19104\\
\texttt{jja56@drexel.edu}}

\noindent
{Department of Mathematics, Drexel University, Philadelphia, PA 19104\\
\texttt{pclarke@math.drexel.edu}}

}

\end{document}